\documentclass[oneside, a4paper,reqno]{amsart}
\usepackage{pdfsync}
\usepackage{stmaryrd}
\usepackage{mathrsfs}
\usepackage{multicol}
\usepackage{amsmath, amsthm, amscd, amssymb, latexsym, eucal}
\usepackage[all]{xy}
 \calclayout \makeatletter
 \calclayout \makeatletter
\def\serieslogo@{} \def\@setcopyright{} \makeatother
\usepackage{multienum}

\usepackage{hyperref}
\usepackage{color}
\usepackage{cite}

\hypersetup{colorlinks=true,
     breaklinks=true,
     linkcolor=blue,    % how to specify linkcolor according to chapter / contents ?
     bookmarks=true,
      pageanchor=true}

\makeatletter
\renewcommand*\env@matrix[1][c]{\hskip -\arraycolsep
  \let\@ifnextchar\new@ifnextchar
  \array{*\c@MaxMatrixCols #1}}
\makeatother

\usepackage{color}

\numberwithin{equation}{section}
\newtheorem{thm}{Theorem}[section]
\newtheorem{cor}[thm]{Corollary}
\newtheorem{lem}[thm]{Lemma}
\newtheorem{prop}[thm]{Proposition}

\newtheorem{Theox}{Theorem}

\theoremstyle{definition}
\newtheorem{defn}[thm]{Definition}
\newtheorem{rem}[thm]{Remark}
\newtheorem{exam}[thm]{Example}

\newtheorem{problem}{Problem}

\newcommand{\lxr}{\longrightarrow}

\newcommand{\A}{\mathscr A}
\newcommand{\B}{\mathscr B}
\newcommand{\C}{\mathscr C}

\newcommand{\D}{\mathscr D}

\newcommand{\F}{\mathcal F}

\newcommand{\T}{\mathcal T}
\newcommand{\U}{\mathcal U}
\newcommand{\V}{\mathcal V}

\newcommand{\X}{\mathcal X}
\newcommand{\Y}{\mathcal Y}

\newcommand{\mL}{\mathsf{L}}

\newcommand{\mt}{\mathsf{T}}
\newcommand{\mh}{\mathsf{H}}
\newcommand{\mU}{\mathsf{U}}
\newcommand{\mz}{\mathsf{Z}}

\newcommand{\mr}{\mathsf{r}}
\newcommand{\mq}{\mathsf{q}}
\newcommand{\mi}{\mathsf{i}}

\newcommand{\ml}{\mathsf{l}}
\newcommand{\me}{\mathsf{e}}

\newcommand{\map}{\mathsf{p}}

 \DeclareMathOperator{\inc}{\mathsf{inc}}
  
\DeclareMathOperator*{\Ker}{\mathsf{Ker}}
 \DeclareMathOperator*{\Image}{\mathsf{Im}}
\DeclareMathOperator*{\Coker}{\mathsf{Coker}}

 \DeclareMathOperator{\pd}{\mathsf{pd}}
\DeclareMathOperator*{\id}{\mathsf{id}}

 \DeclareMathOperator*{\gd}{\mathsf{gl.dim}}

  \DeclareMathOperator*{\maxx}{\mathsf{max}}
 
\DeclareMathOperator*{\Mod}{\mathsf{Mod}-\!}

 \DeclareMathOperator*{\smod}{\mathsf{mod}-\!}

\DeclareMathOperator*{\Inj}{\mathsf{Inj}}
\DeclareMathOperator*{\Proj}{\mathsf{Proj}}

\newcommand{\GProj}{\operatorname{\mathsf{GProj}}\nolimits}
 \newcommand{\Gproj}{\operatorname{\mathsf{Gproj}}\nolimits}
 \newcommand{\GInj}{\operatorname{\mathsf{GInj}}\nolimits}
 \newcommand{\Ginj}{\operatorname{\mathsf{Ginj}}\nolimits}

\DeclareMathOperator*{\add}{\mathsf{add}}

\DeclareMathOperator{\Hom}{\mathsf{Hom}}

\DeclareMathOperator*{\Ext}{\mathsf{Ext}}

  \DeclareMathOperator*{\op}{\mathsf{op}}
   
  \DeclareMathOperator*{\silp}{\mathsf{silp}}
  \DeclareMathOperator*{\spli}{\mathsf{spli}}

   \DeclareMathOperator*{\dime}{\mathsf{dim}}

\DeclareMathOperator*{\Mor}{\mathsf{Mor}}

\newcommand{\iso}{\cong}

\newcommand{\iden}{\operatorname{Id}\nolimits}

\newsavebox{\proofbox}
\savebox{\proofbox}{\begin{picture}(7,7)%
  \put(0,0){\framebox(7,7){}}\end{picture}}

\begin{document}
%\linenumbers

%\setprotcode\font
%    {\it \setprotcode\font}
%    {\bf \setprotcode\font}
%    {\bf \it \setprotcode\font}
%    \pdfprotrudechars=1

\vspace*{-1cm}

\title[]{Ladders of Recollements of Abelian Categories}

\author[]{Nan Gao, Steffen Koenig and Chrysostomos Psaroudakis}
\address{Nan Gao \\ Department of Mathematics, Shanghai University, Shanghai 200444, PR China}
\email{nangao@shu.edu.cn}

\address{Steffen Koenig \\ Institute of Algebra and Number Theory, University of Stuttgart, Pfaffenwaldring 57, 70569 Stuttgart, Germany}
\email{skoenig@mathematik.uni-stuttgart.de}

\address{Chrysostomos Psaroudakis\\
Department of Mathematics, Aristotle University of Thessaloniki, Thessaloniki 54124, Greece}
\email{chpsaroud@math.auth.gr}

\date{\today}

\keywords{Ladders, derived categories, singularity categories, torsion pairs,
Gorenstein categories, Gorenstein projective objects.}

\subjclass[2010]{16E10;16E35;16E65;16G;16G50;16S50;18E;18E30}

\begin{abstract}
Ladders of recollements of abelian categories are introduced, and used to
address three general problems. Ladders of a certain height allow to construct
recollements of triangulated categories, involving derived categories and
singularity categories, from abelian ones. Ladders also allow
to tilt abelian recollements, and ladders guarantee that properties like
Gorenstein projective or injective are preserved by some functors
in abelian recollements. Breaking symmetry is crucial in 
developing this theory.
\end{abstract}

\maketitle

\setcounter{tocdepth}{1} \tableofcontents

\section{Introduction and main results}

Recollements of triangulated or abelian categories
\[
\xymatrix@C=0.5cm{
\A \ar[rrr]^{\mathsf{i}} &&& \B \ar[rrr]^{\mathsf{e}}  \ar @/_1.5pc/[lll]_{\mathsf{q}}  \ar
 @/^1.5pc/[lll]^{\mathsf{p}} &&& \C
\ar @/_1.5pc/[lll]_{\mathsf{l}} \ar
 @/^1.5pc/[lll]^{\mathsf{r}}
 }
\]
can be seen as short exact sequences or semi-orthogonal decompositions, deconstructing a large middle term $\B$
into smaller end terms $\A$ and $\C$. Introduced by Beilinson, Bernstein and Deligne \cite{BBD} for triangulated categories,
recollements have been used to stratify derived categories of sheaves and, following Cline, Parshall and Scott \cite{CPS}, to stratify highest weight
categories in algebraic Lie theory. Recollements of derived categories also are used to provide reduction techniques for homological conjectures,
long exact sequences for homological or K-theoretic invariants and comparisons of homological or K-theoretic data. For module categories of rings, each idempotent $e$ in a ring $B$ provides natural analogues of Grothendieck's six functors, defining recollements of module categories
\[
\xymatrix@C=0.5cm{
\Mod{B/BeB} \ar[rrr]^{\inc} &&& \Mod{B} \ar[rrr]^{e(-)} \ar
@/_1.5pc/[lll]_{B/BeB\otimes_B-}  \ar
 @/^1.5pc/[lll]^{\Hom_B(B/BeB,-)} &&& \Mod{eBe}
\ar @/_1.5pc/[lll]_{Be\otimes_{eBe}-} \ar
 @/^1.5pc/[lll]^{\Hom_{eBe}(eB,-)}
 }
\]
These, and more generally recollements of abelian
categories have been used in various contexts, too (see for instance \cite{Buchweitz, Pira, Psaroud:survey}). There are,
however, big differences between the triangulated and the abelian setup. In particular, by \cite{PsaroudVitoria:1}, all recollements of module
categories are, up to equivalence, given by idempotents, and these recollements usually do not induce recollements of the corresponding
derived categories. More precisely, to be able to construct a recollement
of derived module categories
\[
\xymatrix@C=0.5cm{
\mathsf{D}(\Mod{B/BeB}) \ar[rrr]^{\inc} &&& \mathsf{D}(\Mod{B}) \ar[rrr]^{e(-)} \ar
@/_1.5pc/[lll]_{}  \ar
 @/^1.5pc/[lll]^{} &&& \mathsf{D}(\Mod{eBe})
\ar @/_1.5pc/[lll]_{} \ar
 @/^1.5pc/[lll]^{}
 }
 \]
one has to make the strong assumption that $BeB$ is a
stratifying ideal, that is, the inclusion of $B/BeB$ into $B$ is a homological embedding. Moreover, by deriving abelian recollements one does not obtain, up to equivalence, all triangulated recollements of derived categories (\!\!\cite{AKLY2}). In general, the rings $B$, $A=B/BeB$ and $C=eBe$ may not have much structure in common.

In general, the existence of a triangulated recollement often is difficult to establish and then provides a strong tool. The existence of
an abelian recollement often is easy to establish, but without further assumptions or information it does not provide a strong tool.
The aim of this article is to systematically enhance the definition of abelian recollements by additional data called ladders, which are
sequences of adjoint functors. In contrast to the triangulated situation (see \cite{BGS, AKLY1, QH}) we propose an asymmetric definition. Breaking the symmetry will turn out to be necessary in order to develop the full range and scope of the theory and making it generally applicable. Ladders
of recollements and their heights are defined as follows:

\begin{defn}~\label{ladderab}~\label{defnladder}
Let $\B$ and $\C$ be abelian categories with an adjoint triple between them: \\
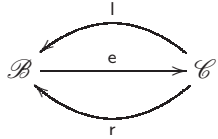
\begin{minipage}{0.2 \textwidth}
\[
\xymatrix@C=0.5cm{
\B \ar[rrr]^{\mathsf{e}}  &&& \C
\ar @/_1.5pc/[lll]_{\mathsf{l}} \ar
 @/^1.5pc/[lll]^{\mathsf{r}} }
\]
\end{minipage}
\vspace*{-0.5cm}
Set $\ml^0:=\ml$ and $\mr^0:=\mr$. A {\bf ladder} is a finite or infinite diagram of additive functors
%\vspace*{-0.5cm}
\hspace*{-1cm}
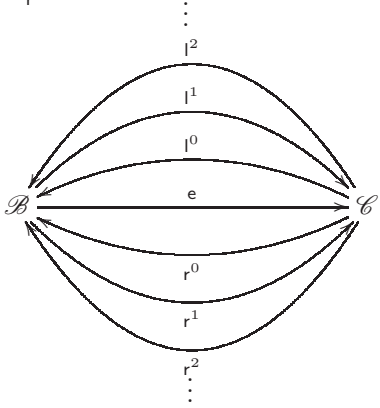
\begin{minipage}{0.4\textwidth}
%\begin{equation}
%\label{ladder}
\xymatrix@C=0.5cm{
 \ \ \ \ \ \ \ \ \ \ \ \ \ \ \ \ \ \ \ \ \  \ \ \  \ \  \ \ \ \ \ \ \ \ \ \ \ \ \vdots &&& \\
 &&&  &&& \\
 \B \ar[rrr]^{\mathsf{e}} \ar @/^3.0pc/[rrr]^{\mathsf{l}^1}   \ar @/_3.0pc/[rrr]_{\mathsf{r}^1}  &&& \C \ar @/_4.5pc/[lll]_{\mathsf{l}^2} \ar @/^4.5pc/[lll]^{\mathsf{r}^2}
\ar @/_1.5pc/[lll]_{\mathsf{l}^0} \ar
 @/^1.5pc/[lll]^{\mathsf{r}^0} \\
 &&& &&& \\
 \ \ \ \ \ \ \ \ \ \ \ \ \ \ \ \ \ \ \ \ \  \ \ \  \ \  \ \ \ \ \ \ \ \ \ \ \ \ \ \vdots &&& \\
 }
%\end{equation}
\end{minipage}
\begin{minipage}{0.47\textwidth}
such that $(\ml^{i+1},\ml^i)$ and $(\mr^{i},\mr^{i+1})$ are adjoint pairs for all $i\geq 0$. We say that the $\ml$-{\bf height} of a ladder is $n$, if there is a tuple $(\ml^{n-1}, \dots, \ml^2,\ml^1,\ml^0)$ of consecutive left adjoints. The $\mr$-{\bf height} of a ladder is defined similarly.
\end{minipage}

\ Let $(\A,\B,\C)$ be a recollement of abelian categories. Set $\ml^0:=\ml$ and $\mr^0:=\mr$. A {\bf ladder} of $(\A,\B,\C)$ is a ladder of the adjoint triple $(\ml, \me,\mr)$ between $\B$ and $\C$, i.e. a finite or infinite diagram of additive functors
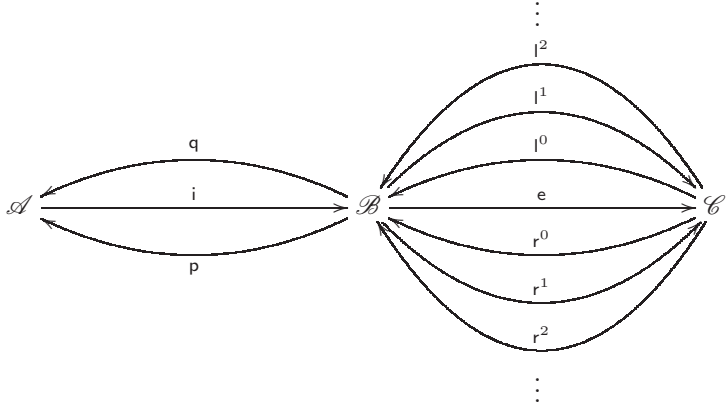
\begin{minipage}{0.7\textwidth}
\[
\xymatrix@C=0.5cm{
 &&& \ \ \ \ \ \ \ \ \ \ \ \ \ \ \ \ \ \ \ \ \  \ \ \  \ \  \ \ \ \ \ \ \ \ \ \ \ \ \vdots &&& \\
 &&&  &&& \\
\A \ar[rrr]^{\mathsf{i}} &&& \B \ar[rrr]^{\mathsf{e}} \ar @/^3.0pc/[rrr]^{\mathsf{l}^1}   \ar @/_1.5pc/[lll]_{\mathsf{q}} \ar @/_3.0pc/[rrr]^{\mathsf{r}^1}  \ar
 @/^1.5pc/[lll]^{\mathsf{p}} &&& \C \ar @/_4.5pc/[lll]_{\mathsf{l}^2} \ar @/^4.5pc/[lll]_{\mathsf{r}^2}
\ar @/_1.5pc/[lll]_{\mathsf{l}^0} \ar
 @/^1.5pc/[lll]_{\mathsf{r}^0} \\
 &&& &&& \\
 &&& \ \ \ \ \ \ \ \ \ \ \ \ \ \ \ \ \ \ \ \ \  \ \ \  \ \  \ \ \ \ \ \ \ \ \ \ \ \ \vdots &&& \\
 }
 \]
\end{minipage}
\begin{minipage}{0.29\textwidth}
The $\ml$-{\bf height} of a ladder of $(\A,\B,\C)$ is the
$\ml$-height of the ladder of the adjoint triple $(\ml,\me, \mr)$.
The $\mr$-{\bf height} of a ladder of $(\A,\B,\C)$ is defined similarly. The {\bf height} of a ladder of $(\A,\B,\C)$ is the sum of the $\ml$-height and the $\mr$-height. The given recollement $(\A,\B,\C)$ then is considered to be a ladder of height one.
\end{minipage}
\end{defn}

In subsection~\ref{subsection:asymmetry}, our reasons for choosing this asymmetry will be explained by comparing the asymmetric ladders introduced here with symmetric ones that turn out to be more limited in their scope.

After collecting basic properties and classes of examples, some of which show already that the length of a ladder in certain recollements is
closely related to homological properties of two-sided ideals in rings, feasibility of this new concept will be demonstrated by addressing three
problems in situations where homological embeddings are not known or not
assumed to exist:

\begin{problem}
Given a recollement of abelian categories enhanced by a left or right ladder of a certain length, is it possible to produce a
recollement of triangulated categories, involving derived or singularity categories? Derived categories and singularity categories
are denoted by $\mathsf{D}$ and $\mathsf{D}_{\textsf{sg}}$, respectively.
\end{problem}

\begin{Theox}[part of Theorem~\ref{mainthmsingladder}]
\label{derivedcat}
Let $(\A,\B,\C)$ be a recollement of abelian categories.
\begin{enumerate}
\item Assume that $(\A,\B,\C)$ has a ladder of $\ml$-height three. Then there
exists a triangle equivalence
\[
\xymatrix{
  \mathsf{D}_{\textsf{sg}}(\B)/\Ker{\ml^1} \ \ar[r]^{ \  \ \ \simeq} & \
  \mathsf{D}_{\textsf{sg}}(\C) }
\]
and  a recollement of triangulated categories
\[
\xymatrix@C=0.5cm{
 \mathsf{D}(\C)\ar[rrr]^{\ml^0} &&& \mathsf{D}(\B) \ar[rrr]^{}  \ar @/_1.5pc/[lll]_{\mathsf{l}^1}  \ar
 @/^1.5pc/[lll]^{\me} &&& \mathsf{D}_{\A}(\B)
\ar @/_1.5pc/[lll]_{} \ar
 @/^1.5pc/[lll]^{}
 }
\]
which restricts to the bounded derived categories.

\item Assume that $(\A,\B,\C)$ has a ladder of $\ml$-height three and $\mr$-height two. Then $(\ml^1, \ml^0, \me)$ induces an adjoint triple between $\mathsf{D}_{\textsf{sg}}(\B)$ and $\mathsf{D}_{\textsf{sg}}(\C)$ and there exists
a recollement of triangulated categories
\[
\xymatrix@C=0.5cm{
 \mathsf{D}_{\textsf{sg}}(\C)\ar[rrr]^{\ml^0} &&& \mathsf{D}_{\textsf{sg}}(\B) \ar[rrr]^{}  \ar @/_1.5pc/[lll]_{\mathsf{l}^1}  \ar
 @/^1.5pc/[lll]^{\me} &&& \Ker{\ml^1}
\ar @/_1.5pc/[lll]_{} \ar
 @/^1.5pc/[lll]^{}
 }
\]

\item Assume that $(\A,\B,\C)$ has a ladder of $\ml$-height four. Then
there exists a recollement of triangulated categories
\[
\xymatrix@C=0.5cm{
\Ker{\ml^1} \ar[rrr]^{} &&& \mathsf{D}_{\textsf{sg}}(\B) \ar[rrr]^{\ml^1}  \ar @/_1.5pc/[lll]_{}  \ar
 @/^1.5pc/[lll]^{} &&& \mathsf{D}_{\textsf{sg}}(\C)
\ar @/_1.5pc/[lll]_{\mathsf{l}^2} \ar
 @/^1.5pc/[lll]^{\ml^0}
 }
\]
\end{enumerate}
\end{Theox}

In (ii) and (iii), the outer terms of the triangulated recollements are
swapping their roles. In both situations, the kernel of a triangulated
functor measures the difference between the two singularity categories.
\bigskip

The second problem uses the following concept of {\em tilt} of an abelian
category:

Let $\C$ be an abelian category with a torsion pair $(\mathcal{T}, \mathcal{F})$. 
Set
\[
\mathcal{H}_{\C}:=\big\{C^{\bullet}\in \mathsf{D}(\mathcal{\C}) \ | \ \mh^{0}(C^{\bullet})\in \mathcal{T}, \mh^{-1}(C^{\bullet})\in \mathcal{F},
\mh^{i}(C^{\bullet})=0, \forall i> 0, \mh^{i}(C^{\bullet})=0, \forall i<-1 \big\}
\]
Then $\mathcal{H}_{\C}$ is called
the Happel-Reiten-Smal\o{} {\em tilt} ({\textsf{HRS}-tilt} 
or just tilt for short) of 
$\C$ by $(\mathcal{T}, \mathcal{F})$, see \cite{HRS}.

\begin{problem} Many abelian categories are derived equivalent, for instance by tilting, and thus occur as hearts of t-structures in the same triangulated
category. No tilting procedure is known that is compatible with abelian recollements. Is it possible to use enhancements by ladders to produce
``tilted"(with respect to torsion theories) recollements both on abelian and on derived level?
\end{problem}

\begin{Theox} \label{tilt}
\ Let $(\A, \B, \C)$ be a recollement of abelian categories which admits a ladder of $\ml$-height three. Assume that
$(\mathcal{T}, \mathcal{F})$ is a torsion pair on $\B$ such that $\ml^{0}\circ \ml^{1}(\mathcal{F})\subseteq \mathcal{F}$ and $\ml^2\circ \ml^1(\T)\subseteq \T$.

\begin{enumerate}
\item $(\ml^{1}(\mathcal{T}), \ml^{1}(\mathcal{F}))$ is a torsion pair on $\C$. We denote by $\mathcal{H}_{\C}$ the  tilt
of $\C$ by $(\ml^{1}(\mathcal{T}), \ml^{1}(\mathcal{F}))$.

\item There exists a recollement of abelian categories\\
\begin{minipage}{0.5\textwidth}
\[
\xymatrix@C=0.5cm{
{\rm S}_{\mathcal H} \ar[rrr]^{} &&& \mathcal{H}_{\B} \ar[rrr]^{\ml^1_{\mathcal{H}}}  \ar @/_1.5pc/[lll]_{}  \ar
 @/^1.5pc/[lll]^{} &&& \mathcal{H}_{\C}
\ar @/_1.5pc/[lll]_{\mathsf{l}^2_{\mathcal{H}}} \ar
 @/^1.5pc/[lll]^{\ml^{0}_{\mathcal{H}}}
 }
\]
\end{minipage} where $\mathcal{S}_{\mathcal H}=\Ker(\ml^1_{\mathcal{H}})$.

\item If the heart $\mathcal{H}_{\C}$ has enough projectives, then there exists a recollement of triangulated categories
\[
\xymatrix@C=0.5cm{
\Ker(\ml^1_{\mathcal{H}}) \ar[rrr]^{} &&& \mathsf{D}(\mathcal{H}_{\B}) \ar[rrr]^{\ml^1_{\mathcal{H}}}  \ar @/_1.5pc/[lll]_{}  \ar
 @/^1.5pc/[lll]^{} &&& \mathsf{D}(\mathcal{H}_{\C})
\ar @/_1.5pc/[lll]_{\mathbb{L}\mathsf{l}^2_{\mathcal{H}}} \ar
 @/^1.5pc/[lll]^{\ml^{0}_{\mathcal{H}}}
 }
\]
\end{enumerate}
\end{Theox}

\begin{problem}
In general, an abelian recollement, for instance of module categories, does not provide connections between homological properties
of the categories, or rings, involved nor between objects and their images under the six functors. Can the enhanced definition by ladders be used to obtain such connections, for instance in terms of Gorenstein homological algebra?
\end{problem}

\begin{Theox}[part of Theorem~\ref{pregproj}]
\label{Gorentein}
Let $(\A,\B,\C)$ be a recollement of abelian categories.
\begin{enumerate}

\item Assume that $(\A,\B,\C)$ has a ladder of $\ml$-height three. Then the functor $\ml^{1}\colon \B\lxr \C$ preserves the property of being Gorenstein and the functor $\me\colon \B\lxr \C$ preserves the property of being Gorenstein injective. Furthermore, if $\B$ is $n$-Gorenstein, then $\C$ is $n$-Gorenstein.

\item Assume that  $(\A,\B,\C)$ admits a ladder of $\ml$-height four. Then the functor $\ml\colon \C\lxr \B$ preserves Gorenstein injective objects. Moreover, $\me\circ\ml\cong {\iden}_{\GInj\C}$.
\end{enumerate}
\end{Theox}

In section 2, some facts about ladders of recollements $(\A,\B,\C)$ of abelian categories will be collected, and the reason for the asymmetry in the definition
will be explained.
In section 3, a number of examples of ladders are presented and projectivity of two-sided ideals will be tested using ladders.
In section 4, ladders of recollements are used to construct recollements of
triangulated categories, proving Theorem~\ref{derivedcat}. In section 5, a technique is provided to produce new torsion pairs in abelian categories via adjoint functors and in particular through Giraud subcategories, and also to provide new recollements of the tilts, proving Theorem~\ref{tilt}. In section 6,
Gorenstein properties are compared using ladders, and Theorem~\ref{Gorentein} is proved.

\subsection*{Conventions and Notation}

For an additive category $\A$, we denote by $\underline{\A}$ the stable category of $\A$. For an additive  functor $F\colon \A \lxr \B$ between additive categories, we denote by $\Image F = \{B \in \B \ | \ B \cong F(A) \ \text{for some} \ A \in \A\}$ the essential image of $F$ and by $\Ker F = \{A \in \A \ | \ F(A) = 0\}$ the kernel of $F$. For an abelian category $\A$, and two classes $\mathcal{X}$ and $\mathcal{Y}$ of objects in $\A$, we put $\mathcal{X}^{\perp}=\{Y\in \A\mid {\rm Ext}^{i}_{\A}(X, Y)=0, \forall \ i>0 \ and \ X\in \mathcal{X}\}$ and $^{\perp}\mathcal{Y}=\{X\in \A\mid {\rm Ext}^{i}_{\A}(X, Y)=0, \ for \ all \ i>0 \ and \ Y\in \mathcal{Y}\}$.

We denote by $\mathsf{D}(\A)$ the derived category of  an abelian category $\A$. Given a recollement of abelian categories  $(\A,\B,\C)$, we denote by $\mathsf{D}_{\A}(\B)$ the full subcategory of $\mathsf{D}(\B)$, whose objects are complexes of objects in $\B$ with cohomologies in $\mi(\A)$.

When considering triangulated categories like derived categories of abelian
categories, existence of these categories always is assumed implicitly.

\subsection*{Acknowledgments} 
This is part of a long-term project of the authors, continuing the first and the third named author's work (\!\!\cite{GaoPsaroudakis2}). Most of this research has been carried out during the first and the third named author's visits to Stuttgart in 2015; work on the project and preparation of 
this article has continued during visits to Stuttgart in the years 
2016, 2017, 2018 and 2019. 
Part of the work has been supported by the National Natural Science
Foundation of China (grant 11771272, held by the first named author) and 
by the German research council DFG through a
postdoctoral position held by the third named author in 2017 and 2018. \\
Some of the results have been announced in the survey article 
\cite{Psaroud:survey}. The article quoted there as [GKP2015] has been extended 
further and split into two articles. The current one is the first of these 
articles. \\
Some of the results of this article and the subsequent one have been presented 
in talks by the first named author at a workshop in Xiamen and at a 
conference in Nagoya.

\section{Definitions and first properties}

\subsection{\bf Recollements}\ Recall the definition of a recollement of abelian categories, see for instance \cite{BBD, Pira, Psaroud:survey}.

\begin{defn}
\label{defnrecolabel}
A \textbf{recollement} between
abelian categories $\A,\B$ and $\C$ is a diagram
\begin{equation}
\label{recollement}
\xymatrix@C=0.5cm{
\A \ar[rrr]^{\mathsf{i}} &&& \B \ar[rrr]^{\mathsf{e}}  \ar @/_1.5pc/[lll]_{\mathsf{q}}  \ar
 @/^1.5pc/[lll]^{\mathsf{p}} &&& \C
\ar @/_1.5pc/[lll]_{\mathsf{l}} \ar
 @/^1.5pc/[lll]^{\mathsf{r}}
 }
\end{equation}
henceforth denoted by $\mathsf{R}_{\mathsf{ab}}(\A,\B,\C)$ or just $(\A,\B,\C)$, satisfying the following conditions$\colon$
\begin{enumerate}
\item $(\mq,\mi,\map)$ and $(\ml,\me,\mr)$ are  adjoint triples.

\item The functors $\mi$, $\ml$, and $\mr$ are fully faithful.

\item $\Image{\mi}=\Ker{\me}$.
\end{enumerate}
\end{defn}

To compare recollements, the definition of equivalence of recollements from
\cite{PsaroudVitoria:1} will be used:

\begin{defn} \label{equivrecs}
Let $(\A,\B,\C)$ and $(\A',\B',\C')$ be two recollements of abelian categories. We say that $(\A,\B,\C)$ and $(\A',\B',\C')$ are {\em equivalent} if there are functors $F\colon \B\lxr \B'$ and $F'\colon \C\lxr \C'$ which are equivalences of categories such that the following diagram commutes up to natural equivalence$\colon$
\[
\xymatrix{
\B \ar[r]^{\me} \ar[d]_{F}^{\cong} & \C \ar[d]^{F'}_{\cong} \\
 \B' \ar[r]^{\me'}  & \C'  \\
}
\]
\end{defn}

{\bf Notation for units and counits.} Throughout,
we denote by  $\mu\colon  \ml\circ\me \lxr \iden_{\B}$, resp. $\kappa\colon \mi\circ\map \lxr \iden_{\B}$, the counit of the adjoint pair $(\ml,\me)$, resp. $(\mi,\map)$, and  by $\lambda \colon
\iden_{\B} \lxr \mi\circ\mq$, resp.  $\nu\colon  \iden_{\B} \lxr \mr\circ\me$, the unit of the adjoint pair $(\mq,\mi)$, resp. $(\me,\mr)$.
\smallskip

Here are some basic properties of functors between abelian categories,
to be used throughout the article: Left adjoints preserve colimits and
thus are right exact. Right adjoints preserve limits and thus are left
exact. Basic properties of functors in abelian recollements, to be used
frequently, are as follows:

\begin{rem}
\label{remdefrecol}
Let $(\A,\B,\C)$ be a recollement of abelian categories.
\begin{enumerate}
\item The functors $\me\colon\B\lxr \C$ and $\mi\colon\A\lxr \B$ are exact.

\item The composition of functors $\mq\circ\ml=\map\circ\mr=0$.

\item The counit $\me\circ\mr \lxr \iden_{\C} $ of the adjoint pair $(\me,\mr)$, the unit $\iden_{\C}\lxr \me\circ\ml $ of the adjoint pair $(\ml,\me)$, the counit $\mq\circ\mi\lxr\iden_{\A}$ of the adjoint pair $(\mq,\mi)$ and the unit $\iden_{\A}\lxr\map\circ\mi$ of the adjoint pair $(\mi,\map)$ are natural isomorphisms.

\item The functor $\mi$ induces an equivalence between $\A$ and the Serre subcategory $\Ker \me = \Image \mi$ of $\B$. In the sequel we shall view this
equivalence as an identification.

 \item For any object $B$ in $\B$, there exist the following exact sequences$\colon$
\begin{equation}
\label{firstcanonicalexactseq}
\xymatrix{
0 \ar[r] & \Ker{\mu_B} \ar[r] & \ml\me(B) \ar[r]^{ \ \mu_B} & B \ar[r]^{\lambda_B \ \ \ } & \mi\mq(B) \ar[r] & 0
}
\end{equation}
\begin{equation}
\label{secondcanonicalexactseq}
\xymatrix{
0 \ar[r] & \mi\map(B) \ar[r]^{ \ \ \kappa_B} & B \ar[r]^{ \nu_B \ \ } & \mr\me(B) \ar[r]^{ } & \Coker{\nu_B} \ar[r] & 0
}
\end{equation}
where $\Ker{\mu_B}$ and $\Coker{\nu_B}$ belong to $\A$.

\item Since the functor $\me$ is exact, it has a fully faithful left adjoint and a fully faithful right adjoint.

\item By the previous claim, $\A$ is a localising and colocalising subcategory of $\B$ and there is an equivalence $\B/\A \simeq \C$. In particular any recollement $\textsf{R}_{\mathsf{ab}}(\A,\B,\C)$ induces a short exact sequence  of abelian categories $ 0 \lxr \A \lxr \B \lxr \C \lxr 0 $. For more details see \cite[subsection~2.1]{Psaroud:survey}.
\end{enumerate}
\end{rem}

\begin{rem}
\label{remadjointtriplerecol}
A recollement is determined by the adjoint triple $(\ml,\me,\mr)$ on the right
hand side, where $\me$ is exact and its left and its right adjoint both are
fully faithful.
Slightly more general, a recollement can be constructed from an adjoint triple
as follows (see \cite[Remark 2.5]{Psaroud:survey} for details):
\\
\begin{minipage}{0.3 \textwidth}
Let $\me\colon\B\lxr \C$ be an exact functor between abelian categories such that there is an adjoint triple as follows:
\end{minipage}
\begin{minipage}{0.7 \textwidth}
\begin{equation}
\label{adjointtriple}
\xymatrix@C=0.5cm{
\B \ar[rrr]^{\mathsf{e}}  &&& \C
\ar @/_1.5pc/[lll]_{\mathsf{l}} \ar
 @/^1.5pc/[lll]^{\mathsf{r}} }
\end{equation}
\end{minipage}
Assume that $\ml$ is fully faithful. The functor $\mq\colon \B\lxr \Ker{\me}$
is defined by $\mq(B)=\Coker{\mu_B}$ where $\mu\colon \ml\circ\me\lxr \iden_{\B}$ is the counit of the adjoint pair $(\ml,\me)$. Similarly, the functor $\map\colon \B\lxr \Ker{\me}$ is defined by $\map(B)=\Ker{\nu_B}$ where $\nu\colon\iden_{\B}\lxr \mr\circ\me$ is the unit of the adjoint pair $(\me,\mr)$.
Then $(\mq,\mi,\map)$ is an adjoint triple, where $\mi\colon \Ker \me \lxr \B$ is the inclusion functor. Hence,
\\
\begin{minipage}{0.5 \textwidth}
\[
\xymatrix@C=0.5cm{
\Ker\me \ar[rrr]^{\mathsf{i}} &&& \B \ar[rrr]^{\mathsf{e}}  \ar @/_1.5pc/[lll]_{\mathsf{q}}  \ar
 @/^1.5pc/[lll]^{\mathsf{p}} &&& \C
\ar @/_1.5pc/[lll]_{\mathsf{l}} \ar
 @/^1.5pc/[lll]^{\mathsf{r}}
 }
\]
\end{minipage}
\begin{minipage}{0.5 \textwidth}
is a recollement of abelian categories. In a similar way, the recollement
can be reconstructed under the assumption that $\mr$ is fully faithful.
\end{minipage}.

It remains to compare the recollement just constructed with the given one.
Let $\mathsf{R}_{\mathsf{ab}}(\A,\B,\C)$ be a recollement and $\mathsf{R}_{\mathsf{ab}}(\Ker{\me},\B,\C)$ the recollement just constructed from the adjoint triple on the right hand side. Then the two recollements $\mathsf{R}_{\mathsf{ab}}(\A,\B,\C)$ and $\mathsf{R}_{\mathsf{ab}}(\Ker{\me},\B,\C)$ are equivalent in the sense of Definition~\ref{equivrecs}.

So, up to equivalence the original recollement
$\mathsf{R}_{\mathsf{ab}}(\A,\B,\C)$ can be reconstructed from the adjoint
triple $(\ml,\me,\mr)$. Thus, there is an alternative way of defining a
recollement of abelian categories: Given abelian categories $\B$ and $\C$, a
recollement with middle term $\A$ is an adjoint triple
$(\ref{adjointtriple})$ such that the functor $\ml$ (or $\mr$) is fully
faithful.
\end{rem}

\subsection{\bf Ladders} \ A ladder of an abelian recollement in the sense
of Definition \ref{ladderab} yields further recollements in the following
way: Assume that there is a ladder as in \ref{ladderab}. Moreover, assume
that $\ml^0$ is fully faithful. Then $\ml^2$, $\ml^4$, $\ldots$ are fully
faithful and $\mr^0$, $\mr^2$, $\ldots$ are also fully faithful.

\begin{prop}
\label{proppropertiesladder}
Let $(\A,\B,\C)$ be a recollement of abelian categories.
\begin{enumerate}
\item  Assume that the recollement $(\A,\B,\C)$ admits a ladder of $\ml$-height $n$ where $n$ is an even positive number. Then there exist recollements of abelian categories
\[
\xymatrix@C=0.5cm{
\Ker{\ml^1} \ar[rrr]^{} &&& \B \ar[rrr]^{\ml^1}  \ar @/_1.5pc/[lll]_{}  \ar
 @/^1.5pc/[lll]^{} &&& \C
\ar @/_1.5pc/[lll]_{\mathsf{l}^2} \ar
 @/^1.5pc/[lll]^{\ml^0}
 }, \, \ldots \, , \xymatrix@C=0.5cm{
\Ker{\ml^{n-1}} \ar[rrr]^{} &&& \B \ar[rrr]^{\ml^{n-1}}  \ar @/_1.5pc/[lll]_{}  \ar
 @/^1.5pc/[lll]^{} &&& \C
\ar @/_1.5pc/[lll]_{\mathsf{l}^{n}} \ar
 @/^1.5pc/[lll]^{\ml^{n-2}}
 }
\]

\item  Assume that the recollement $(\A,\B,\C)$ admits a ladder of $\mr$-height $n$ where $n$ is an even positive number. Then there exist recollements of abelian categories
\[
\xymatrix@C=0.5cm{
\Ker{\mr^1} \ar[rrr]^{} &&& \B \ar[rrr]^{\mr^1}  \ar @/_1.5pc/[lll]_{}  \ar
 @/^1.5pc/[lll]^{} &&& \C
\ar @/_1.5pc/[lll]_{\mathsf{r}^0} \ar
 @/^1.5pc/[lll]^{\mr^2}
 }, \, \ldots \, , \xymatrix@C=0.5cm{
\Ker{\mr^{n-1}} \ar[rrr]^{} &&& \B \ar[rrr]^{\mr^{n-1}}  \ar @/_1.5pc/[lll]_{}  \ar
 @/^1.5pc/[lll]^{} &&& \C
\ar @/_1.5pc/[lll]_{\mathsf{r}^{n-2}} \ar
 @/^1.5pc/[lll]^{\mr^{n}}
 }
\]

\end{enumerate}
\end{prop}

\begin{proof}
Statements (i) and (ii) follow from Remark~\ref{remadjointtriplerecol}.
\end{proof}

\subsection{\bf The reason for asymmetry}
\label{subsection:asymmetry}
Definition~\ref{defnladder} may be unexpected due to its lack
  of symmetry. A symmetric definition analogous to that used for ladders
of triangulated categories
  (see for instance \cite[Section~3]{AKLY1}) appears much more natural, at
  least at first sight. We will see however that using such a definition
  would severely restrict the scope and range of the theory we are
  attempting to build and of the applications exemplified by the main
  theorems of this article. Symmetric ladders only can exist for comma
categories, which in the case of module categories means that the ring in
the middle of the recollement has to be triangular.
Moreover, a special case of a classification
  result by Feng and Zhang \cite{FZ} implies that for our purposes a symmetric
  definition of ladders has little potential to distinguish various situations
as it is needed in the
  characterisations we are aiming at.
  We first address the issue of range and
scope by showing that symmetric ladders occur only in special situations. 

Suppose there is a diagram of the following form, 
which we will call a
{\em symmetric ladder}: \\
\begin{minipage}{0.65\textwidth}
\[
\xymatrix@C=0.5cm{
 &&& \vdots \ \ \ \ \ \ \ \ \ \ \ \ \ \ \ \ \ \ \ \ \ \ \ \ \ \ \  \ \ \  \ \  \ \ \ \ \ \ \ \ \vdots &&& \\
 &&&  &&& \\
\A \ar[rrr]^{\mathsf{i}} \ar @/_3.0pc/[rrr]^{{\map^1}} \ar @/^3.0pc/[rrr]^{\mq^1}  &&& \B \ar @/_4.5pc/[lll]_{\mq^2}\ar[rrr]^{\mathsf{e}} \ar @/^3.0pc/[rrr]^{\mathsf{l}^1}   \ar @/_1.5pc/[lll]_{\mathsf{q}} \ar @/_3.0pc/[rrr]^{\mathsf{r}^1}  \ar
 @/^1.5pc/[lll]_{\mathsf{p}} \ar @/^4.5pc/[lll]_{\map^2} &&& \C \ar @/_4.5pc/[lll]_{\mathsf{l}^2} \ar @/^4.5pc/[lll]_{\mathsf{r}^2}
\ar @/_1.5pc/[lll]_{\mathsf{l}^0} \ar
 @/^1.5pc/[lll]_{\mathsf{r}^0} \\
 &&& &&& \\
 &&& \vdots \ \ \ \ \ \ \ \ \ \ \ \ \ \ \ \ \ \ \ \ \ \ \ \ \ \ \  \ \ \  \ \  \ \ \ \ \ \ \ \ \vdots &&& \\
 }
\]
\end{minipage}
\begin{minipage}{0.35 \textwidth}
In such a situation, the adjoint triple $(\mq^1, \mq, \mi)$ guarantees that the functor $\mq$ is exact, and similarly the adjoint triple $(\mi,\map,\map^1)$ shows that the functor $\map$ is exact. We will prove below: if there is a recollement situation $(\A,\B,\C)$ such that the functor $\mq$ or $\map$ is exact then the recollement $(\A,\B,\C)$ is equivalent to one given by a comma category. In case of module categories over rings, the latter statement says that the middle category is modules over a triangular matrix ring. Thus, such symmetric ladders
can occur in very special situations only.
\end{minipage}
\smallskip
\begin{minipage}{0.7 \textwidth}
We now recall what is a comma category. Let $G\colon \B\lxr \A$ be a right exact functor between abelian categories. The objects of the {\em comma category} $\C=(G \downarrow \iden)$ are triples $(A,B,f)$ where $f\colon G(B)\lxr A$ is a morphism in $\A$. A morphism $\gamma=(\alpha, \beta)\colon (A,B,f)\lxr (A',B',f')$ in $\C$ consists of two morphisms $\alpha\colon A\lxr A'$ in $\A$ and $\beta\colon B\lxr B'$ in $\B$ such that the following diagram is commutative$\colon$
\end{minipage}
\begin{minipage}{0.3 \textwidth}
\[
\xymatrix{
  G(B) \ar[d]_{G(\beta)} \ar[r]^{f} & A \ar[d]^{\alpha}     \\
  G(B')    \ar[r]^{f'} & A'                  }
\]
\end{minipage}
Since the functor $G$ is right exact, it follows from \cite{FGR} that the comma category $\C$ is abelian. We define the following functors$\colon$
\begin{enumerate}
\item The functor $\mt_{\B}\colon \B\lxr \C$ is defined by $\mt_{\B}(Y)=(G(Y),Y,\iden_{GY})$ on objects $Y$ in $\B$ and given a morphism $\beta\colon Y\lxr Y'$ in $\B$ then $\mt_{\B}(\beta)=(G(\beta),\beta)$ is a morphism in $\C$.

\item The functor $\mU_{\B}\colon \C\lxr \B$ is defined by $\mU_{\B}(A,B,f)=B$ on objects $(A,B,f)$ in $\C$ and given a morphism $(\alpha,\beta)\colon (A,B,f)\lxr (A',B',f')$ in $\C$ then $\mU_{\B}(\alpha,\beta)=\beta$ is a morphism in $\B$. Similarly, the functor $\mU_{\A}\colon \C\lxr \A$ is defined.

\item The functor $\mz_{\B}\colon \B\lxr \C$ is defined by $\mz_{\B}(Y)=(0,Y,0)$ on objects $Y$ in $\B$ and given a morphism $\beta\colon Y\lxr Y'$ in $\B$ then $\mz_{\B}(\beta)=(0,\beta)$ is a morphism in $\C$. Similarly the functor $\mz_{\A}\colon \A\lxr \C$ is defined.

\item The functor $\mq\colon \C\lxr \A$ is defined by $\mq(A,B,f)=\Coker{f}$ on objects $(A,B,f)$ in $\C$. A morphism ${(\alpha,\beta)\colon (A,B,f)\lxr
(A',B',f')}$ in $\C$ induces a morphism
${\mq(\alpha,\beta)\colon \Coker{f}\lxr \Coker{f'}}$.
\end{enumerate}
When $G$ has a right adjoint $G'\colon \A\lxr \B$, there are more functors. We denote by $\epsilon\colon GG'\lxr \iden_{\A}$ the counit and by $\eta\colon\iden_{\B}\lxr G'G$ the unit of the adjoint pair $(G,G')$.

\begin{enumerate}
\item The functor $\mh_{\A}\colon\A\lxr \C$ is defined by $\mh_{\A}(X)=(X,G'(X),\epsilon_X)$ on objects $X$ in $\A$  and given a morphism $\alpha\colon X\lxr X'$ in $\A$ then $\mh_{\A}(\alpha)=(\alpha,G'(\alpha))$ is a morphism in $\C$.

\item The functor $\map\colon \C\lxr \B$ is defined by $\map(A,B,f)=\Ker{(\eta_B\circ G'(f))}$ on objects $(A,B,f)$ in $\C$. A morphism $(\alpha,\beta)\colon (A,B,f)\lxr
(A',B',f')$ in $\C$ induces
$\map(\alpha,\beta)\colon {\Ker{(\eta_B\circ G'(f))}\lxr \Ker{(\eta_{B'}\circ G'(f'))}}$.
\end{enumerate}
It is easy to check, see also \cite{Psaroud:survey}, that the diagrams$\colon$
\begin{minipage}{0.5 \textwidth}
\begin{equation}
\label{firstrecolcommacat}
\xymatrix@C=0.5cm{
\A \ar[rrr]^{\mz_{\A}} &&& \C
\ar[rrr]^{\mU_{\B}} \ar @/_1.5pc/[lll]_{\mq} \ar
 @/^1.5pc/[lll]^{\mU_{\A}} &&& \B
\ar @/_1.5pc/[lll]_{\mt_{\B}} \ar
 @/^1.5pc/[lll]^{\mz_{\B}}
 }
\end{equation}
\end{minipage}
\begin{minipage}{0.5 \textwidth}
and
\[
\xymatrix@C=0.5cm{
\B \ar[rrr]^{\mz_{\B}} &&& \C
\ar[rrr]^{\mU_{\A}} \ar @/_1.5pc/[lll]_{\mU_{\B}} \ar
 @/^1.5pc/[lll]^{\map} &&& \A
\ar @/_1.5pc/[lll]_{\mz_{\A}} \ar
 @/^1.5pc/[lll]^{\mh_{\A}}
 }
\]
\end{minipage}
are recollements of abelian categories.

The following result is due to Franjou-Pirashvili
  \cite[Proposition~8.9]{Pira} who proved it using a characterisation of when
  a recollement of abelian categories is equivalent to the
  MacPherson--Vilonen recollement. We provide a direct proof using the
  recollement structure of a comma category together with a characterisation
  by Franjou-Pirashvili for a comparison functor between recollements to be
  an equivalence.

\begin{prop}
\label{proprecolcommacatFPresult}
Let $(\A,\B,\C)$ be a recollement of abelian categories. Assume that the functor $\map$ is exact and that $\B$ and $\C$ have enough projective objects.  Let $(\map\ml \downarrow \iden_{\C})$ be the comma category whose objects are triples of the form $(A,C,f)$ where $A\in \A$, $C\in \C$ and $f\colon \map \ml(C)\lxr A$ is a morphism in $\A$. Then the recollements of abelian categories $(\A,\B,\C)$ and $(\A, (\iden \downarrow G), \C)$ are equivalent.

In particular, if the recollement 
$(\A,\B,\C)$ admits a ladder of $\mr$-height at least two,
then the functor $\map$ is exact. If in addition $\B$ and $\C$ have enough 
projective objects, then $\A$ is equivalent to a comma category. \\
In particular, if $(\A,\B,\C)$ is a recollement of module categories, then 
the ring in the middle is triangular.
\begin{proof}
Since the functor $\map$ is exact, the composition $\map \ml$ is right exact and therefore the comma category $(\map \ml \downarrow \iden)$ is abelian \cite{FGR}. The objects are triples $(A,C,f)$ where $f\colon \map \ml(C)\lxr A$ is a morphism in $\A$. Then as in \textnormal{(\ref{firstrecolcommacat})}, there is a recollement $\mathsf{R}_{\mathsf{ab}}(\A, (\map \ml \downarrow \iden), \C)$. We claim that the recollements $\mathsf{R}_{\mathsf{ab}}(\A,\B,\C)$ and $\mathsf{R}_{\mathsf{ab}}(\A, (\map \ml \downarrow \iden), \C)$ are equivalent in the sense of Definition~\ref{equivrecs}. To show this, we define the functor $\F\colon \B\lxr (\map \ml \downarrow \iden)$ by $\F(B)=(\map(B), \me(B), \map \mu_B)$ on objects $B\in \B$, and if $b\colon B\lxr B'$ is a morphism in $\B$, then $\F(b)=(\map(b), \me(b))$ is a morphism in $(\map \ml \downarrow \iden)$. Then it follows immediately that $\F$ is a comparison functor, i.e.\  the following diagram commutes with all the structural functors of the recollements$\colon$
\begin{equation}\nonumber
\xymatrix@C=0.5cm{
\A \ar@{=}[ddd]_{} \ar[rrr]^{\mi} &&& \B \ar[ddd]_{\F} \ar[rrr]^{\me}  \ar @/_1.5pc/[lll]_{\mq}  \ar @/^1.5pc/[lll]^{\map} &&& \C \ar@{=}[ddd]
\ar @/_1.5pc/[lll]_{\ml} \ar
 @/^1.5pc/[lll]^{\mr}   \\
 &&& &&& \\
 &&& &&&  \\
\A  \ar[rrr]^{\mz_{\A}} &&& (\map\ml \downarrow \iden_{\C}) \ar[rrr]^{\mU_{\C}}  \ar @/_1.5pc/[lll]_{\mq'}  \ar
 @/^1.5pc/[lll]^{\mU_{\A}} &&& \C
\ar @/_1.5pc/[lll]_{\mt_{\C}} \ar
 @/^1.5pc/[lll]^{\mz_{\C}}
 }
\end{equation}
\begin{minipage}{0.6 \textwidth}
Note also that the functor $\F$ is exact since the functors $\map$ and $\me$ are exact. It remains now to show that $\F$ is an equivalence of categories. For this, it suffices to show that $\F$ is left admissible (see \cite[Theorem~7.2]{Pira}), i.e.\ the following diagram is commutative$\colon$
\end{minipage}
\begin{minipage}{0.4 \textwidth}
\begin{equation}
\label{leftadmissible}
\xymatrix{
\A \ar@{=}[d]_{} & \Ker\map \ar[l]_{\mL_1\mq} \ar[d]^{\F} \\
 (\map \ml \downarrow \iden)  & \Ker\mU_{\A} \ar[l]_{ \ \ \ \mL_1\mq'}  \\
}
\end{equation}
\end{minipage}

Let $B$ be an object of $\B$ such that $\map(B)=0$. Then clearly $\F(B)$ lies in
$\Ker\mU_{\A}$. Consider now a short exact sequence with $Q$ in $\Proj\C\colon$
\begin{equation}
\label{shortexactseqeb}
\xymatrix{
0 \ar[r]_{} & \Omega(\me(B)) \ar[r]^{ \ \ \ \beta} & Q \ar[r]^{\alpha \ \ \ } & \me(B) \ar[r] & 0 }
\end{equation}
Then there is a short exact sequence
\begin{equation}
\label{shortexactseqtcq}
\xymatrix{
0 \ar[r]_{} & \Ker{(0, \alpha)} \ar[r]^{} & \mt_{\C}(Q) \ar[r]^{(0, \alpha) \ \ } & Z_{\B}(\me(B)) \ar[r] & 0 }
\end{equation}
where $\mt_{\C}(Q)$ lies in $\Proj(\map\ml \downarrow \iden_{\C})$ and $\Ker{(0, \alpha)}=(\map\ml(Q), \Ker{\alpha}, \map\ml(\beta))$. Recall that the functor $\mq'$ sends a triple $(A,C,f)$ to the object $\Coker{f}$. Applying the functor $\mq'$ to $(\ref{shortexactseqtcq})$, we get that the object
$\mL_1\mq'\mz_{\B}(\me(B))$ is isomorphic to $\map\ml\me(B)$.

We now compute the first left derived functor $\mL_1\mq(B)$. Applying the exact functor $\map$ to $(\ref{firstcanonicalexactseq})$ and since $\map(B)=0$, it follows that $\mq(B)=0$ and therefore the counit map $\mu_B\colon \ml\me(B)\lxr B$ is an epimorphism. Then applying the functor $\ml$ to $(\ref{shortexactseqeb})$ yields the short exact sequence
\[
\xymatrix{
0 \ar[r]_{} & \Omega(B) \ar[r]^{} & \ml(Q) \ar[rr]^{\mu_B\circ \ml(\alpha)} && \me(B) \ar[r] & 0 }
\]
with $\ml(Q)$ projective. Applying the functor $\mq$ and using $\mq\ml=0$ (Remark~\ref{remdefrecol}) gives an isomorphism $\mL_1\mq(B)\cong \mq(\Omega(B))$. Consider the following exact commutative diagram$\colon$
\[
\xymatrix{
 & \map\ml\me(\Omega(B)) \ar[d] \ar[r]^{} & \map\ml(Q) \ar@{=}[d] \ar[rr]^{\map\ml(\alpha)} && \map\ml\me(B) \ar@{->>}[d]^{\map\mu_B} \ar[r]  & 0  \\
 0 \ar[r] & \map(\Omega(B)) \ar@{->>}[d] \ar[r]^{} & \map\ml(Q) \ar[rr]^{\map(\mu_B\circ\ml(\alpha))} && \map(B)\ar[r]  & 0  \\
 & \mq(\Omega(B))  &  &&   &  }
\]
Since $\map(B)=0$, the Snake Lemma implies that $\mq(\Omega(B))$ is isomorphic to $\map\ml\me(B)$. Hence, the diagram $(\ref{leftadmissible})$ is commutative and thus the functor $\F$ is an equivalence of categories.

Finally, if $(\A,\B,\C)$ is a recollement of module categories, then the above comma category is the module category of a triangular matrix ring, see \cite{ARS, FGR}.
\end{proof}
\end{prop}

The following result is dual and has a similar proof.

\begin{prop}
\label{proprecolcommacatFPresultdual}
Let $(\A,\B,\C)$ be a recollement of abelian categories. Assume that the functor $\mq$ is exact and that $\B$ and $\C$ have enough injective objects. Let $(\iden_{\C} \downarrow \mq\mr)$ be the comma category whose objects are triples of the form $(A,C,f)$ where $A\in \A$, $C\in \C$ and $f\colon A\lxr \mq\mr(C)$ is a morphism in $\A$. Then the recollements of abelian categories $(\A,\B,\C)$ and $(\A, (\iden_{\C} \downarrow \mq\mr), \C)$ are equivalent.

In particular, if the recollement 
$(\A,\B,\C)$ admits a ladder of $\ml$-height at least two,
then the functor $\mq$ is exact. If in addition $\B$ and $\C$ have enough 
injective objects, then $\A$ is equivalent to a comma category. \\
In particular, if $(\A,\B,\C)$ is a recollement of module categories, then 
the ring in the middle is triangular.
\end{prop}

\begin{rem}
Propositions~\ref{proprecolcommacatFPresult}
and~\ref{proprecolcommacatFPresultdual} show that non-trivial symmetric
ladders of module categories (or more general abelian categories)
 only can exist in the case of comma
categories or triangular matrix algebras. In fact, when the symmetric ladder
extends the given recollement downwards, the functors $\mathsf{p}$ and
$\mathsf{r}$ must be exact and Proposition~\ref{proprecolcommacatFPresult}
becomes applicable. When the symmetric ladder extends the given recollement
upwards, the functors $\mathsf{q}$ and $\mathsf{l}$ must be exact and
Proposition~\ref{proprecolcommacatFPresultdual} becomes applicable.

Another limitation of the concept of symmetric ladders is implied by work of
Feng and Zhang \cite{FZ}. Starting with a Serre subcategory of a
Grothendieck category and the corresponding exact sequence of abelian
categories, they have given a full classification of all symmetric
partial or full recollements or ladders. This classification gives just
seven cases, three of which are partial recollements. The fourth case is
recollements that cannot be extended to non-trivial ladders. When non-trivial
symmetric ladders exist, Feng and Zhang's classification states that there
are only three cases: Upwards extension by one step or downwards extension by
one step or ladders that are infinite both upwards and downwards.

Hence, using symmetric ladders severely restricts the scope and range of the
theory by limiting it to comma categories or triangular matrix rings, and in
addition by allowing for only three kinds of non-trivial ladders, which is
much less flexibility than we need for homological characterisations such
as in the main results of this article.
\end{rem}

\section{More examples and some ladders in action}

For various classes of rings, ladders of recollements are constructed and
ladders (and their heights) will be connected to ring theoretical or module
theoretical properties. Ladders determine such properties and the existence
of ladders depends on properties of certain modules. In the last
subsection, the height of a ladder is characterised in terms of certain
modules being projective or not (Theorem \ref{charheight}).

\subsection{\bf Morita context rings} Any ring with a decomposition of
the unit into a sum of two orthogonal idempotents can be written as a
Morita context ring.

\begin{exam}
\label{tofDelta}
Let $R$ be a ring and consider the Morita context ring
$\Delta_{(0,0)}=\bigl(\begin{smallmatrix}
R & R \\
R & R
\end{smallmatrix}\bigr)$ (see \cite{GP, GaoPsaroudakis}).
Its modules are tuples of the form
\[
(X,Y,f,g)\colon \xymatrix@C=0.5cm{
X \ar@<-.7ex>[rr]_{f} && Y \ar@<-.7ex>[ll]_-{g}}
\]
where $X$, $Y$ are in $\Mod{R}$ and $g\circ f=0=f\circ g$. A morphism between two tuples $(X,Y,f,g)$ and $(X',Y',f',g')$ is a pair of $R$-homomorphisms $(a,b)$ such that $b\circ f=f'\circ a$ and $a\circ g= g'\circ b$. By \cite[Proposition~4.4]{GaoPsaroudakis2}, the module category $\Mod\Delta_{(0,0)}$ admits a recollement of module categories with an infinite ladder (of period three).

This algebra is the preprojective algebra of Dynkin type $\mathbb{A}_2$. By
\cite[Proposition~4.4]{GaoPsaroudakis2}, there are infinite ladders (of
period three) for all preprojective algebras of Dynkin type $\mathbb{A}_n$ and
more generally for the preprojective algebras of Dynkin species
$\mathbb{A}_n$.
\end{exam}

\subsection{\bf Homological embeddings}
An exact functor $\mi\colon\A\lxr \B$ between abelian categories is called a
{\em homological embedding} (see \cite{Psaroud:homolrecol}), if the map ${\mi_{X,Y}^n\colon\Ext_{\A}^n(X,Y)\lxr \Ext_{\B}^n(\mi(X),\mi(Y))}$ is an isomorphism of abelian groups for all $X, Y$ in $\A$ and for all $n\geq 0$.
A recollement $(\A, \B, \C)$ of abelian categories is called a {\em
homological recollement}, if $\mi$ is a homological embedding.
\\ 
Let $\Lambda$ be an associative ring and $I$ a two-sided ideal of $\Lambda$.
We are going to construct a family of homological recollements, depending
on a natural number $n \geq 2$, which we fix from now on. Define an
 $n\times n$ matrix ring $\Gamma$ and idempotents $f,g \in \Gamma$ 

\begin{minipage}{0.3 \textwidth}
\[ 
\Gamma=
         \begin{pmatrix}
           \Lambda & I & I & \cdots & I & I\\
           \Lambda &  \Lambda & I & \cdots  & I & I\\
           \Lambda & \Lambda &  \Lambda  & \cdots  & I & I\\
 \vdots & \vdots  & \vdots & \ddots & \vdots & \vdots \\
           \Lambda & \Lambda & \Lambda  & \cdots  & \Lambda & I\\
           \Lambda & \Lambda & \Lambda  & \cdots  & \Lambda & \Lambda\\
           \end{pmatrix}
 \]
\end{minipage}
\begin{minipage}{0.1 \textwidth}
and
\end{minipage}
\begin{minipage}{0.5 \textwidth}
 \[
 f=\begin{pmatrix}
 0 & 0 & \cdots & 0\\
 0 & 0 & \cdots  & 0\\
 \vdots & \vdots   & \ddots & \vdots  \\
  0 & 0 & \cdots  & 1\\
\end{pmatrix} \ \ \ \text{and}  \ \ \ g=\begin{pmatrix}
 1 & 0 & \cdots & 0\\
 0 & 0 & \cdots  & 0\\
 \vdots & \vdots   & \ddots & \vdots  \\
  0 & 0 & \cdots  & 0\\
\end{pmatrix}
\]
\end{minipage}

\begin{minipage}{0.35\textwidth}
and an $(n-1)\times (n-1)$ matrix ring
\end{minipage}
\begin{minipage}{0.35\textwidth}
\[
\Sigma=\begin{pmatrix}
          \Lambda/I & 0 & \cdots & 0 \\
           \Lambda/I & \Lambda/I & \cdots & 0 \\
           \vdots & \vdots & \ddots  & \vdots \\
           \Lambda/I & \Lambda/I & \cdots & \Lambda/I
         \end{pmatrix}
 \]
\end{minipage}

There are isomorphisms $\Sigma=\Gamma/\Gamma
f\Gamma$ and $\Lambda=f\Gamma f$. Then there is a homological recollement of
module categories $(\Mod{\Sigma},\Mod{\Gamma},\Mod{\Lambda})$, which as we
will see has $r$-height at least three and $l$-height at least one
\begin{equation}
\label{recolmatrixalg}
\xymatrix@C=0.5cm{
 \Mod\Sigma \ar[rrr]^{\mi} &&& \Mod\Gamma  \ar[rrr]^{f\Gamma\otimes_{\Gamma}-} \ar @/_1.5pc/[lll]_{\mq}  \ar @/^1.5pc/[lll]^{\map} \ar @/_3.0pc/[rrr]^{\mr^1}  &&& \Mod\Lambda \ar @/^4.5pc/[lll]_{\mr^2}
 \ar @/^1.5pc/[lll]_{\mr^0} \ar @/_1.5pc/[lll]_{\ml^0}
 } %\phantom{xxxxxxxxxx} (\dag)
\end{equation}
%\vspace*{-1cm}
where
\[
\left\{
  \begin{array}{lll}
   \me=f\Gamma\otimes_{\Gamma}- \cong f(-) & \hbox{} \\
   %        & \hbox{} \\
   \mr^0=\Hom_{f\Gamma f}(f\Gamma,-)\cong \Gamma g\otimes_\Lambda-  & \hbox{} \\
   % & \hbox{} \\
   \mr^1=\Hom_{\Gamma}(\Gamma g,-)\iso g\Gamma\otimes_{\Gamma}-\iso g(-)  & \hbox{} \\
   % & \hbox{} \\
   \mr^2=\Hom_{\Lambda}(g\Gamma,-)  & \hbox{}
  \end{array}
\right.
\]

The values of the $l$-height and the $r$-height depend on properties of
the ideal $I$:

\begin{prop} \label{ideal}
Let $\Lambda$, $\Gamma$ and $\Sigma$ as above. The following hold.
\begin{enumerate}
\item The recollement $(\ref{recolmatrixalg})$ is homological and it has $l$-height at least one and
$r$-height at least three.

\item If $I$ is not projective as both a left and a right $\Lambda\mbox{-}$module, then the recollement $(\ref{recolmatrixalg})$ has exactly $l$-height one and $r$-height
three.

\item If $_{\Lambda}I$ is projective, then the recollement $(\ref{recolmatrixalg})$ has $l$-height at least one and $r$-height at least four.

\item If $I_{\Lambda}$ is projective, then the recollement $(\ref{recolmatrixalg})$ has $l$-height at least two and $r$-height at least three.

\item \label{idealderived}
The recollement $(\ref{recolmatrixalg})$ induces a recollement of derived module
categories which admits a ladder of height at least four if and only
 $_{\Lambda}I$ has finite projective dimension.
\end{enumerate}
\end{prop}

\begin{proof}
The fact that $(\ref{recolmatrixalg})$ is homological can be checked directly by using that $f\Gamma$ is projective and since $\Gamma f\otimes_{f\Gamma f} f\Gamma\cong \Gamma f\Gamma$ (i.e.\ $\Gamma f\Gamma$ is a stratifying ideal), see also \cite{GaoKoenigPsaroudakis} for a more detailed proof. For the ladder the key point is the description of $\mr^0$: The functor $\mr^0$ is exact since the left $\Lambda$-module $f\Gamma=(\begin{smallmatrix}
\Lambda & \Lambda & \cdots & \Lambda
\end{smallmatrix})$ is projective. Also, the functor $\mr^0$ preserves coproducts since $f\Gamma$ is finitely generated. Then, by Watts' Theorem, the functor $\mr^0$ is naturally isomorphic to $_\Gamma\Hom_{\Lambda}(f\Gamma, \Lambda)\otimes_\Lambda-$. Moreover, $\Hom_{\Lambda}(f\Gamma, \Lambda)$ is isomorphic to $\Gamma g$ as $\Gamma$-$\Lambda$-bimodules. This completes the description of $\mr^0$. Thus, the functor $\mr^0$ becomes the left adjoint of the standard adjoint triple induced by the idempotent $g$ and there is  a ladder of $\mr$-height at least three.

Next we have to ask if $\mr^2$ admits a right adjoint so that $(\Mod{\Sigma},\Mod{\Gamma},\Mod{\Lambda})$ has $\mr$-height at least four. We compute that $g\Gamma = (\begin{smallmatrix}
\Lambda & I & \cdots & I
\end{smallmatrix})$ and therefore $\mr^2$ admits a right adjoint if and only if $_\Lambda I$ is projective. Similarly, the functor $\ml^0=\Gamma f\otimes_\Lambda-$ has a left adjoint if and only if $I_\Lambda$ is projective. In this case, the recollement $(\Mod{\Sigma},\Mod{\Gamma},\Mod{\Lambda})$ has $\ml$-height at least two.

The recollement $(\dag)$  induces a ladder of derived module categories (\!\!\cite{CPS:memoirs}, see also \cite[Theorem~8.3]{Psaroud:survey}) of
height at least three:
\[
\xymatrix@C=0.5cm{
\mathsf{D}(\Mod\Sigma) \ar[rrr]^{\mi} &&& \mathsf{D}(\Mod\Gamma)  \ar[rrr]^{f\Gamma\otimes_{\Gamma}-} \ar @/_1.5pc/[lll]_{}  \ar @/^1.5pc/[lll]^{} \ar @/_3.0pc/[rrr]^{\mr^1}  &&& \mathsf{D}(\Mod\Lambda) \ar @/^4.5pc/[lll]_{\mathbb{R}\mr^2}
 \ar @/^1.5pc/[lll]_{\mr^0} \ar @/_1.5pc/[lll]_{\mathbb{L}\ml^0}
 }
\]
Note that the adjoints on the right side of the recollement induce adjoints on the left side, so we get a ladder of $\mr$-height three (going downwards).

We infer that $_{\Lambda}I$ has finite projective dimension if and only if  there exists a bounded complex $P^{\bullet}$ of projective left $\Lambda\mbox{-}$modules such that the functor $\mathbb{R}\mr^2\cong \Hom_{\mathsf{D}(\Mod\Lambda)}(P^{\bullet}, -)$ if and only if $\mathbb{R}\mr^2$ admits a right adjoint.
\end{proof}

\begin{exam}
\label{nonsym}
Let $\Lambda$ be an Artin algebra and $\Lambda e\Lambda$  a stratifying ideal of $\Lambda$. Recall from \cite{CPS:memoirs}, see also \cite{AKLY2}, that $\Lambda e\Lambda$ is called stratifying if the surjective homomorphism $\Lambda\lxr \Lambda/\Lambda e\Lambda$ is homological \cite{GeigleLenzing}, i.e.\ the canonical functor $\Mod\Lambda/\Lambda e\Lambda\lxr \Mod\Lambda$ is a homological embedding.   
Let\\
\begin{minipage}{0.35\textwidth}
\[
\Gamma=
         \begin{pmatrix}
           \Lambda & \Lambda e\Lambda & \Lambda e\Lambda & \cdots & \Lambda e\Lambda & \Lambda e\Lambda\\
           \Lambda &  \Lambda & \Lambda e\Lambda & \cdots  & \Lambda e\Lambda & \Lambda e\Lambda\\
           \Lambda & \Lambda &  \Lambda  & \cdots  & \Lambda e\Lambda & \Lambda e\Lambda\\
 \vdots & \vdots  & \vdots & \ddots & \vdots & \vdots \\
           \Lambda & \Lambda & \Lambda  & \cdots  & \Lambda & \Lambda e\Lambda\\
           \Lambda & \Lambda & \Lambda  & \cdots  & \Lambda & \Lambda\\
           \end{pmatrix}
 \]
\end{minipage}
\begin{minipage}{0.1\textwidth}
\phantom{xxxxx}
\end{minipage}
\begin{minipage}{0.2\textwidth}
be an $n\times n$ matrix algebra and let
\end{minipage}
\begin{minipage}{0.35\textwidth}
 \[
\Sigma=\begin{pmatrix}
          \Lambda/\Lambda e\Lambda & 0 & \cdots & 0 \\
           \Lambda/\Lambda e\Lambda & \Lambda/\Lambda e\Lambda & \cdots & 0 \\
           \vdots & \vdots & \ddots  & \vdots \\
           \Lambda/\Lambda e\Lambda & \Lambda/\Lambda e\Lambda & \cdots & \Lambda/\Lambda e\Lambda
         \end{pmatrix}
 \]
\end{minipage}\\
be an $(n-1)\times (n-1)$ matrix algebra.
 Taking an idempotent $f=\begin{pmatrix}
 0 & 0 & \cdots & 0\\
 0 & 0 & \cdots  & 0\\
 \vdots & \vdots   & \ddots & \vdots  \\
  0 & 0 & \cdots  & 1\\
\end{pmatrix}$ and $g=\begin{pmatrix}
 1 & 0 & \cdots & 0\\
 0 & 0 & \cdots  & 0\\
 \vdots & \vdots   & \ddots & \vdots  \\
  0 & 0 & \cdots  & 0\\
\end{pmatrix}$ of $\Gamma$, then by Proposition~\ref{ideal} there exists a homological recollement of module categories, which has $l$-height at least one and
$r$-height at least three:
\[
\xymatrix@C=0.5cm{
\smod\Sigma \ar[rrr]^{\mi} &&& \smod\Gamma  \ar[rrr]^{f\Gamma\otimes_{\Gamma}-} \ar @/_1.5pc/[lll]_{\mq}  \ar @/^1.5pc/[lll]^{\map} \ar @/_3.0pc/[rrr]^{\mr^1}  &&& \smod\Lambda \ar @/^4.5pc/[lll]_{\mr^2}
 \ar @/^1.5pc/[lll]_{\mr^0} \ar @/_1.5pc/[lll]_{\ml^0}
 }
\]
\end{exam}

Ladders can be used to identify idempotent ideals as stratifying ideals:

\begin{prop}\ Let $\Lambda$ be an Artin algebra and $I=\Lambda e\Lambda$ an idempotent ideal of $\Lambda$. Let\\
\begin{minipage}{0.35\textwidth}
\[
\Gamma=
         \begin{pmatrix}
           \Lambda & \Lambda e\Lambda & \Lambda e\Lambda & \cdots & \Lambda e\Lambda & \Lambda e\Lambda\\
           \Lambda &  \Lambda & \Lambda e\Lambda & \cdots  & \Lambda e\Lambda & \Lambda e\Lambda\\
           \Lambda & \Lambda &  \Lambda  & \cdots  & \Lambda e\Lambda & \Lambda e\Lambda\\
 \vdots & \vdots  & \vdots & \ddots & \vdots & \vdots \\
           \Lambda & \Lambda & \Lambda  & \cdots  & \Lambda & \Lambda e\Lambda\\
           \Lambda & \Lambda & \Lambda  & \cdots  & \Lambda & \Lambda\\
           \end{pmatrix}
 \]
\end{minipage}
\begin{minipage}{0.1\textwidth}
\phantom{xxxxx}
\end{minipage}
\begin{minipage}{0.2\textwidth}
be an $n\times n$ matrix algebra and let
\end{minipage}
\begin{minipage}{0.35\textwidth}
\[
\Sigma=\begin{pmatrix}
          \Lambda/\Lambda e\Lambda & 0 & \cdots & 0 \\
           \Lambda/\Lambda e\Lambda & \Lambda/\Lambda e\Lambda & \cdots & 0 \\
           \vdots & \vdots & \ddots  & \vdots \\
           \Lambda/\Lambda e\Lambda & \Lambda/\Lambda e\Lambda & \cdots & \Lambda/\Lambda e\Lambda
         \end{pmatrix}
 \]
\end{minipage} \\
be an $(n-1)\times (n-1)$ matrix algebra. Consider the following recollement:
 \[
\xymatrix@C=0.5cm{
\smod\Sigma \ar[rrr]^{\mi} &&& \smod\Gamma  \ar[rrr]^{f\Gamma\otimes_{\Gamma}-} \ar @/_1.5pc/[lll]_{\mq}  \ar @/^1.5pc/[lll]^{\map} &&& \smod\Lambda \ar @/^1.5pc/[lll]_{\mr^0} \ar @/_1.5pc/[lll]_{\ml^0}
 }    \  \ \ \ \ \ \ \ \  (*)
\]
If $(*)$ has $l$-height two or $r$-height three, then $\Lambda e\Lambda$ is a stratifying ideal of $\Lambda$.
\begin{proof} \ Apply Proposition~\ref{nonsym}: If $(*)$ has $l$-height two or $r$-height three,
then $\Gamma f$ or $f\Gamma$ is a projective $f\Gamma f\mbox{-}$module. This means that $\Lambda e\Lambda$ is projective as a right
$\Lambda\mbox{-}$module or as a left $\Lambda\mbox{-}$module. Thus in both
cases, $\Lambda e\Lambda$ is a stratifying ideal.
\end{proof}
\end{prop}

Given an ideal $I$, one may form another kind of algebras also
yielding ladders:

\begin{exam} Let $A$ be a $k$-algebra, where $k$ is a commutative ring, and $I$ a two-sided ideal of $A$. Consider the following matrix rings
\[
\Lambda =\begin{pmatrix}
           A & I & I^{2} & I^{3} & I^{4}\\
           A & A & I & I^{3} & I^{4}\\
           A & A & A  & I & I^{4}\\
            A & A & A  & A & I\\
           A & A & A  & A & A\\
\end{pmatrix} \ \ \ \ \text{and} \ \ \ \ \Gamma = \begin{pmatrix}
           A/I^{4} & I/I^{4} & I^{2}/I^{4} & I^{3}/I^{4}\\
           A/I^{4} & A/I^{4} & I/I^{4} & I^{3}/I^{4}\\
           A/I^{4} & A/I^{4} & A/I^{4} & I/I^{4}\\
           A/I & A/I & A/I  &  A/I\\
                   \end{pmatrix}
\]
and the idempotent elements $e=\begin{pmatrix}
 0 & 0 & \cdots & 0\\
 0 & 0 & \cdots  & 0\\
 \vdots & \vdots   & \ddots & \vdots  \\
  0 & 0 & \cdots  & 1\\
\end{pmatrix}$ and $f=\begin{pmatrix}
1 & 0 & \cdots & 0\\
 0 & 0 & \cdots  & 0\\
 \vdots & \vdots   & \ddots & \vdots  \\
  0 & 0 & \cdots  & 0\\
\end{pmatrix}$.
Then the following recollement of module categories $(\Mod{\Lambda/\Lambda e\Lambda},\Mod{\Lambda},\Mod{e\Lambda e})$ has $\mr$-height at least three \\
\begin{minipage}{0.6 \textwidth}
\[
\xymatrix@C=0.5cm{
\Mod{\Gamma} \ar[rrr]^{} &&& \Mod{\Lambda} \ar[rrr]^{e} \ar @/_1.5pc/[lll]_{}  \ar @/^1.5pc/[lll]^{} \ar @/_3.0pc/[rrr]^{\mr^1}  &&& \Mod{A} \ar @/^4.5pc/[lll]_{\mr^2}
 \ar @/^1.5pc/[lll]_{\mr^0} \ar @/_1.5pc/[lll]_{}
 }
\]
\end{minipage}
\begin{minipage}{0.3 \textwidth}
where $\me=e\Lambda\otimes_{\Lambda}-$, \ $\mr^0\cong\Lambda f\otimes_A-$, $\mr^1=\Hom_{\Lambda}(\Lambda f,-)\iso f\Lambda\otimes_{\Lambda}-\iso f(-)$ and $\mr^2=\Hom_{A}(f\Lambda,-)$.
\end{minipage}
\smallskip
\\
Indeed, $f\Lambda f\cong e\Lambda e\cong A$, $\Lambda/\Lambda e\Lambda\cong \Gamma$, \ ${\rm Hom}_{A}(e\Lambda,-)\cong\Lambda f\otimes_{A}-$ and $f\Lambda=\bigl(\begin{matrix}
 A & I & I^{2} & I^{3} & I^{4}\\
 \end{matrix}\bigr)$. Moreover, $\mr^{2}$ admits a right adjoint if and only if each $I^{i}$ is a projective left $A$-module for all $1\leq i\leq 4$.
\\
In particular, let $k$ be a field, $A=k[x]/\langle x^{n}\rangle$ for some $n\geq 1$ and $I={\rm rad}A$. Then the recollement  has $\mr$-height three, since  $I={\rm rad}A$ is a non-projective maximal ideal of $A$.
\end{exam}

Now we turn to examples, where finitely many ideals are given:

\begin{exam}
\label{mul}
Let $A$ be a $k$-algebra over a commutative ring $k$, and $I_{1}, \ I_{2}, \ldots,$ $I_{n-1}$  two-sided ideals of $A$ such that $I_{n-1}\subseteq I_{i}$ for any $1\leq i\leq n-2$. Consider the following matrix rings \\
\begin{minipage}{0.35 \textwidth}
\[
\Lambda =
         \begin{pmatrix}
           A & I_{1} & I_{2} & \cdots & I_{n-2} & I_{n-1}\\
           A & A & I_{2} & \cdots  & I_{n-2} & I_{n-1}\\
           A & A & A  & \cdots  & I_{n-2} & I_{n-1}\\
 \vdots & \vdots  & \vdots & \ddots & \vdots & \vdots \\
           A & A & A  & \cdots  & A & I_{n-1}\\
           A & A & A  & \cdots  & A & A\\
           \end{pmatrix}
\]
\end{minipage}
\begin{minipage}{0.5 \textwidth}
\phantom{xxxx}
\end{minipage}
\begin{minipage}{0.1 \textwidth}
and
\end{minipage}
\begin{minipage}{0.4 \textwidth}
\[
\Gamma = \begin{pmatrix}
           A/I_{n-1} & I_{1}/I_{n-1} & I_{2}/I_{n-1} & \cdots & I_{n-2}/I_{n-1}\\
           A/I_{n-1} & A/I_{n-1} & I_{2}/I_{n-1} & \cdots  & I_{n-2}/I_{n-1}\\
           \vdots & \vdots  & \vdots & \ddots & \vdots\\
           A/I_{n-1} & A/I_{n-1} & A/I_{n-1}  & \cdots  & A/I_{n-1}\\
                   \end{pmatrix}
\]
\end{minipage} \\
and the idempotent elements $e=\begin{pmatrix}
 0 & 0 & \cdots & 0\\
 0 & 0 & \cdots  & 0\\
 \vdots & \vdots   & \ddots & \vdots  \\
  0 & 0 & \cdots  & 1\\
\end{pmatrix}$ and $f=\begin{pmatrix}
1 & 0 & \cdots & 0\\
 0 & 0 & \cdots  & 0\\
 \vdots & \vdots   & \ddots & \vdots  \\
  0 & 0 & \cdots  & 0\\
\end{pmatrix}$.
Then the following recollement of module categories $(\Mod{\Lambda/\Lambda e\Lambda},\Mod{\Lambda},\Mod{e\Lambda e})$ has $\mr$-height at least three\\
\begin{minipage}{0.5 \textwidth}
\[
\xymatrix@C=0.5cm{
\Mod{\Gamma} \ar[rrr]^{} &&& \Mod{\Lambda} \ar[rrr]^{e} \ar @/_1.5pc/[lll]_{}  \ar @/^1.5pc/[lll]^{} \ar @/_3.0pc/[rrr]^{\mr^1}  &&& \Mod{A} \ar @/^4.5pc/[lll]_{\mr^2}
 \ar @/^1.5pc/[lll]_{\mr^0} \ar @/_1.5pc/[lll]_{\ml^{0}}
 }
\]
\end{minipage}
\begin{minipage}{0.45 \textwidth}
where $\ml^0=\Lambda e\otimes_A-$, \ $\me=e\Lambda\otimes_{\Lambda}-$, \  $\mr^0=\Lambda f\otimes_A-$, $\mr^1=\Hom_{\Lambda}(\Lambda f,-)\iso f\Lambda\otimes_{\Lambda}-\iso f(-)$ and $\mr^2=\Hom_{A}(f\Lambda,-)$.
\end{minipage}

In particular, if $A$ is a hereditary algebra, then the recollement has $\ml$-height at least two and $\mr$-height at least four.
Indeed, $f\Lambda f\cong e\Lambda e\cong A$, \ $\Lambda/\Lambda e\Lambda\cong \Gamma$, \  ${\rm Hom}_{A}(e\Lambda,-)\cong\Lambda f\otimes_{A}-$, \  $f\Lambda\cong\bigl(\begin{matrix}
 A & I_{1} & \cdots & I_{n-1}\\
 \end{matrix}\bigr)$ and $\Lambda e=\begin{pmatrix}
 I_{n-1} \\
 I_{n-1}\\
 \cdots \\
 I_{n-1}\\
 A
 \end{pmatrix}$. Thus $\mr^{2}$ admits a right adjoint if and only if each $I_{i}$ is a projective left $A$-module for all $1\leq i\leq n-1$, and
$\ml^{0}$ admits a left adjoint if and only if $I_{n-1}$ is a projective right $A$-module.

For an explicit example, let $A$ be hereditary, then each $I_{i}$ is a projective left $A$-module for all $1\leq i\leq n-1$. Thus $\mr^{2}$ admits a right adjoint and
$\ml^{0}$ admits a left adjoint.
\end{exam}

\begin{rem} Example~\ref{mul} shows that derived equivalences don't preserve the height of ladders. Indeed, by \cite[Corollary 3.2]{Cy},
the ring $\Lambda$ is derived equivalent to the matrix ring
\[
\Delta:=\begin{pmatrix}
           A/I_{1} & 0 & 0 & \cdots & 0 & 0\\
           A/I_{1} & A/I_{2} & 0 & \cdots  & 0 & 0\\
           A/I_{1} & A/I_{2} & A/I_{3}  & \cdots  & 0 & 0\\
 \vdots & \vdots  & \vdots & \ddots & \vdots & \vdots \\
            A/I_{1} & A/I_{2} & A/I_{3}  & \cdots  & A/I_{n-1} & 0\\
          A/I_{1} & A/I_{2} & A/I_{3}   & \cdots  & A/I_{n-1} & A\\
           \end{pmatrix}
\]
\noindent Let $\Sigma:=\begin{pmatrix}
           A/I_{1} & 0 & 0 & \cdots & 0 \\
           A/I_{1} & A/I_{2} & 0 & \cdots  & 0 \\
           A/I_{1} & A/I_{2} & A/I_{3}  & \cdots  & 0 \\
 \vdots & \vdots  & \vdots & \ddots & \vdots \\
            A/I_{1} & A/I_{2} & A/I_{3}  & \cdots  & A/I_{n-1} \\
     \end{pmatrix}$ and the idempotent element $e:=\begin{pmatrix}
 0 & 0 & \cdots & 0\\
 0 & 0 & \cdots  & 0\\
 \vdots & \vdots   & \ddots & \vdots  \\
  0 & 0 & \cdots  & 1\\
\end{pmatrix}$.  Then the following recollement $(\Mod{\Delta/\Delta e\Delta},$ $\Mod{\Delta},\Mod{e\Delta e})$ has $\ml$-height at least two$\colon$
\[
\xymatrix@C=0.5cm{
\Mod{\Sigma} \ar[rrr]^{\mi} &&& \Mod{\Delta}\ar @/^3.0pc/[rrr]^{\ml^1} \ar[rrr]^{\me} \ar @/_1.5pc/[lll]_{\mq}  \ar @/^1.5pc/[lll]^{\map} &&& \Mod{A}
\ar @/_1.5pc/[lll]_{\ml^0} \ar
 @/^1.5pc/[lll]_{\mr^0}
 }
\]

Indeed, $e\Delta e\cong A$, \ $\Delta/\Delta e\Delta\cong \Sigma$, \ $\Delta e\cong\begin{pmatrix}
 0 \\
 0 \\
 \vdots  \\
  A\\
\end{pmatrix}$ and $e\Delta e\cong A$. This implies that $\Delta e$ is a projective right $e\Delta e$-module and so $\ml^{0}$ admits a left adjoint.
If $I_{n-1}$ is not a projective right $A$-module, then $(\Mod \Gamma, \Mod \Lambda, \Mod A)$ has $\ml$-height one. This means that the two recollements $(\Mod \Gamma, \Mod \Lambda, \Mod A)$ and $(\Mod \Sigma, \Mod \Delta, $ $\Mod A)$ have different $\ml$-heights.
\end{rem}

\subsection{\bf Morphism categories}\ Let $\A$ be an abelian category and $\Mor_n(\A)$ the $n$-morphism category of $\A$ (see \cite{ARS}).
The objects of $\Mor_n(\A)$ are sequences of the form $X_{1}\stackrel{f_1}{\lxr} X_{2}\stackrel{f_2}{\lxr}\cdots \lxr X_{n-1}\xrightarrow{f_{n-1}} X_{n}$ such that all $X_{i}\in \A$; this object is denoted
by $(X, f)$. Given two objects $(X, f)$ and $(X', f')$, a morphism between them is a commutative diagram$\colon$
\[
\xymatrix{
 X_1 \ar[d]^{a_1} \ar[r]^{f_1 \ } & X_2 \ar[d]^{a_2} \ar[r]^{f_2 \ } & \cdots \ar[r]^{f_{n-1} \ } & X_{n} \ar[d]^{a_{n}} \\
 X_1' \ar[r]^{f_1' \ } & X_2' \ar[r]^{f_2' \ } & \cdots \ar[r]^{f_{n-1}' \ } & X_{n}' }
\]
that is,  $f_{i}'a_i=a_{i+1}f_{i}$ for all $1\leq i\leq n-1$, where $a_i\colon A_i\lxr A_i'$ are morphisms in $\A$ for all $1\leq i\leq n$. The category
$\Mor_n(\A)$ is known to be an abelian category.

For an example, let $R$ be a ring and consider the lower triangular $n\times n$-matrix ring
\[
\mathsf{T}_n(R) = \begin{pmatrix}
           R & 0 & \cdots & 0 \\
           R & R & \cdots & 0 \\
           \vdots & \vdots & \ddots  & \vdots \\
           R & R & \cdots & R
         \end{pmatrix}
\]
Then there is an equivalence of abelian categories between $\Mod\mathsf{T}_n(R)$ and $\Mor_n(\Mod{R})$, see \cite{ARS}.

Define functors from $\Mor_n(\A)$ to $\A\colon$
\[
\left\{
  \begin{array}{lll}
   \ml^1(X_1\stackrel{f_1}{\lxr} X_2\stackrel{f_2}{\lxr} \cdots \xrightarrow{f_{n-1}} X_n)=\Coker{f_{n-1}}  & \hbox{} \\
           & \hbox{} \\
   \me(X_1\lxr X_2\lxr \cdots \lxr X_n)=X_n  & \hbox{} \\
   & \hbox{} \\
   \mr^1(X_1\lxr X_2\lxr \cdots \lxr X_n)=X_1  & \hbox{} \\
   & \hbox{} \\
   \mr^3(X_1\stackrel{f_1}{\lxr} X_2\stackrel{f_2}{\lxr} \cdots \xrightarrow{f_{n-1}} X_n)=\Ker{f_1}  & \hbox{}
  \end{array}
\right.
\]
On morphisms these functors are defined in a natural way. Define functors from $\A$ to $\Mor_n(\A)\colon$
\[
\left\{
  \begin{array}{lll}
   \ml^0(X)=(0\lxr 0\lxr \cdots \lxr X)  & \hbox{} \\
           & \hbox{} \\
   \mr^0(X)=(X\xrightarrow{\iden_{X}} X\xrightarrow{\iden_{X}} \cdots \xrightarrow{\iden_{X}} X)  & \hbox{} \\
   & \hbox{} \\
   \mr^2(X)=(X\lxr 0\lxr \cdots \lxr 0)  & \hbox{}
  \end{array}
\right.
\]
Moreover, define the functor
\[
\mi\colon {\Mor}_{n-1}(\A)\lxr {\Mor}_n(\A), \ \mi(X_1\lxr X_2\lxr \cdots \lxr X_{n-1})=(X_1\lxr X_2\lxr \cdots \lxr X_{n-1}\lxr 0)
\]
and finally define functors from $\Mor_{n}(\A)$ to  $\Mor_{n-1}(\A)\colon$
\[
\left\{
  \begin{array}{lll}
   \mq(X_1\lxr X_2\lxr \cdots \lxr X_n)=(X_1\lxr X_2\lxr \cdots \lxr X_{n-1})  & \hbox{} \\
           & \hbox{} \\
   \map(X_1\xrightarrow{f_{1}}  \cdots \xrightarrow{f_{n-1}} X_n)=(\Ker{(f_{n-1}\cdots f_2 f_1)}\lxr\Ker{(f_{n-1}\cdots f_2)}\lxr  \cdots \lxr\Ker{f_{n-1}}) & \hbox{} \\
           & \hbox{}
  \end{array}
\right.
\]

\begin{exam}
\label{extension}
Let $\Mor_n(\A)$ be the $n$-morphism category of an abelian category $\A$.
Then the functors defined above fit into a recollement of abelian categories $(\Mor_{n-1}(\A), \Mor_{n}(\A), \A)$ with $\ml$-height two and $\mr$-height four$\colon$
\[
\xymatrix@C=0.5cm{
\Mor_{n-1}(\A) \ar[rrr]^{\mi} &&& \Mor_{n}(\A) \ar @/_3.0pc/[rrr]^{\mr^1} \ar @/_7.5pc/[rrr]^{\mr^3} \ar @/^3.0pc/[rrr]^{\ml^1} \ar[rrr]^{\me} \ar @/_1.5pc/[lll]_{\mq}  \ar @/^1.5pc/[lll]^{\map} &&& \A \ar @/^4.5pc/[lll]_{\mr^2}
\ar @/_1.5pc/[lll]_{\ml^0} \ar
 @/^1.5pc/[lll]_{\mr^0}
 }
\]

\textsf{Claim} $1\colon$  $(\Mor_{n-1}(\A), \Mor_{n}(\A), \A)$ is a recollement.

By the definition of the functor $\me$, $\Ker\me$ is equivalent to $\Mor_{n-1}(\A)$. It suffices to prove that $(\ml^{0}, \me, \mr^{0})$ is an adjoint triple and $\ml^{0}$ is fully faithful.

Let $X$ be an object in $\A$ and $(Y, g)=(Y_1\xrightarrow{g_1}Y_2\xrightarrow{g_2} \cdots \xrightarrow{g_{n-2}} Y_{n-1}\xrightarrow{g_{n-1}} Y_{n})$ an object in $\Mor_{n}(\A)$. Since $\Hom_{\Mor_{n}(\A)}(\ml^{0}(X), (Y, g))\cong \Hom_{\A}(X, Y_{n})=\Hom_{\A}(X, \me(Y,g))$, it follows that
$(\ml^{0}, \me)$ is an adjoint pair between $\Mor_n(\A)$ and $\A$. Since the object $\mr^0(X)$ has an identity, there are isomorphisms $\Hom_{\Mor_{n}(\A)}((Y, g), \mr^{0}(X))\cong \Hom_{\A}(Y_{n}, X)=\Hom_{\A}(\me(Y, g), X)$. This shows that $(\me, \mr^{0})$ is an adjoint pair between $\A$ and $\Mor_{n}(\A)$. Since $\me\circ\ml^{0}(X)=X$ for each object $X$ in $\A$, it follows that $\ml^{0}$ is fully faithful.

\textsf{Claim} $2\colon$ $(\ml^1, \ml^0)$, $(\mr^0, \mr^1)$, $(\mr^1, \mr^2)$ and $(\mr^2, \mr^3)$ are adjoint pairs.

Let $(X,f)$ be an object in $\Mor_{n}(\A)$ and $Y$ an object in $\A$. The
isomorphisms ${\Hom_{\Mor_n(\A)}((X,f), \ml^0(Y))}$ $\cong {\Hom_{\A}(\Coker f_{n-1}, Y)} = \Hom_{\A}(\ml^1(X,f), Y)$ imply that $(\ml^1, \ml^0)$ is an adjoint pair.
Moreover, the isomorphisms ${\Hom_{\Mor_n(\A)}((X,f), \mr^2(Y))}\cong \Hom_{\A}(X_1, Y)=\Hom_{\A}(\mr^1(X,f), Y)$ imply that $(\mr^1, \mr^2)$ is an adjoint pair.
Moreover, $(\mr^0, \mr^1)$ and $(\mr^2, \mr^3)$ are adjoint pairs. This implies that the recollement $(\Mor_{n-1}(\A), \Mor_{n}(\A), \A)$ admits a ladder of $\ml$-height two and $\mr$-height four.

Note that $\ml^1$ and $\mr^3$ are not in general exact functors.
\end{exam}

\begin{rem} 
Consider the $n$-morphism category $\Mor_n(\A)$ and its recollement.
\begin{enumerate}
\item Since the functor $\mq$ is exact, the functor $\mi$ is a homological
embedding by \cite[Theorem $3.9$]{Psaroud:homolrecol}.

\item A special case of Example~\ref{extension} is about
module categories: The recollement of module categories $(\Mod{\mathsf{T}_{n-1}(R)}, \Mod{\mathsf{T}_{n}(R)}, \Mod{R})$ admits a ladder of $\ml$-height two and $\mr$-height four. This relates to previous examples involving homological
embeddings as follows:

Let $e=\begin{pmatrix}
 0 & 0 & \cdots & 0\\
 0 & 0 & \cdots  & 0\\
 \vdots & \vdots   & \ddots & \vdots  \\
  0 & 0 & \cdots  & 1\\
\end{pmatrix}$ and $f=\begin{pmatrix}
1 & 0 & \cdots & 0\\
 0 & 0 & \cdots  & 0\\
 \vdots & \vdots   & \ddots & \vdots  \\
  0 & 0 & \cdots  & 0\\
\end{pmatrix}$. Then the above functor $\me=e\mathsf{T}_n(R)$, the module $e\mathsf{T}_n(R)$ is a left projective $e\mathsf{T}_n(R)e$-module,
$\Hom_{e\mathsf{T}_n(R)e}(e\mathsf{T}_n(R),$ $ \ e\mathsf{T}_n(R)e)\cong \mathsf{T}_n(R)f$ as $(\mathsf{T}_n(R), e\mathsf{T}_n(R)e)$-bimodules,
and moreover, $\mr^{3}=\Hom_{\mathsf{T}_n(R)}(\Hom_{f\mathsf{T}_n(R)f}(f\mathsf{T}_n(R), f\mathsf{T}_n(R)f), -)$.

Here, $\Hom_{f\mathsf{T}_n(R)f}(f\mathsf{T}_n(R), f\mathsf{T}_n(R)f)=(R, 0, \cdots, 0)$, which is not a left projective $\mathsf{T}_n(R)$-module. This implies that $\mr^{3}$ is not an exact functor.
\end{enumerate}
\end{rem}

\subsection{Characterising the height of a ladder} A ladder  of r-height or l-height $n$  can be built up inductively by going up or going down step by step for some integer $n\geq 2$. A characterisation is given when a recollement of module categories admits a ladder of $\mr$-height or $\ml$-height exactly $m$ for $m\geq 2$.

\begin{thm} \label{charheight}
Let $\Lambda$ be an Artin algebra, $e$ an idempotent element\footnote{To avoid trivial cases in both these two theorems (i.e.\ $\Gamma$ being zero or $\Lambda$), the idempotent $e$ is considered to be non-trivial.} 
and $\Gamma:=e \Lambda e$. Consider the recollement of $\Mod{\Lambda}$ induced by
the idempotent element $e$. Define a sequence of $\Lambda$-$\Gamma$ \textnormal{(or $\Gamma$-$\Lambda$)}-bimodules
by $M_0:=e\Lambda$, $M_1:=\Hom_\Gamma(M_0,\Gamma)$,
$M_2:=\Hom_\Lambda(M_1,\Lambda)$, \dots, $M_{2n+1}:=\Hom_\Gamma(M_{2n},\Gamma)$,
$M_{2n+2}:=\Hom_\Lambda(M_{2n+1},\Lambda)$, \dots (for $n \geq 0$). Then:
\begin{enumerate}
\item
The recollement admits a ladder of $r$-height exactly $2n+2$ if and only if
$M_j$ is projective as a left $\Gamma$-module for all even $j \leq 2n$,
$M_j$ is projective as a left $\Lambda$-module for all odd $j < 2n+1$ and
$M_{2n+1}$ is not projective as a left $\Lambda$-module.
\item
The recollement admits a ladder of $r$-height exactly $2n+3$ if and only if
$M_j$ is projective as a left $\Gamma$-module for all even $j \leq 2n$, $M_{2n+2}$ is not projective as a left $\Gamma$-module and
$M_j$ is projective as a left $\Lambda$-module for all odd $j \leq 2n+1$.
\end{enumerate}
\end{thm}

\begin{proof}
The ladder can be built up inductively by going down step by step. When a new
functor $r^{j+1}$ appears at the bottom, it is a right adjoint and thus left
exact. Moreover, then $r^j$, which is already known to be left exact, is a left
adjoint of $r^{j+1}$ and thus right exact, hence exact.

Since all modules $M_j$ are finitely generated over $\Lambda$ or $\Gamma$,
respectively, all functors $\Hom(M_j,-)$ preserve coproducts. Thus, Watts'
theorem can be applied to identify all $r^j$ inductively as such functors.
Therefore, $r^{j+1}$ is exact if and only $M_j$ is a projective left module over
$\Lambda$ or $\Gamma$, respectively. Building up the ladder stops exactly
when the functor $r^{j+1}$ at the bottom is not exact, which means that the
module $M_j$ occuring in its first argument is not projective.
\end{proof}

We close this subsection by formulating the dual result. In this case, to get the adjoints instead of Watts' theorem we use the well known isomorphism $P\otimes_{\Lambda}-\cong \Hom_{\Lambda}(\Hom_{\Lambda}(P,\Lambda),-)$ for a finitely generated projective $\Lambda$-module $P$. The easy proof is left to the reader.

\begin{thm} \label{charheight2}
Let $\Lambda$ be an Artin algebra, $e$ an idempotent element and
$\Gamma:=e \Lambda e$. Consider the recollement of $\Mod{\Lambda}$ induced by
the idempotent element $e$. Define a sequence of $\Lambda$-$\Gamma$ \textnormal{(or $\Gamma$-$\Lambda$)}-bimodules
by $M_0:=\Lambda e$, $M_1:=\Hom_\Gamma(M_0,\Gamma)$,
$M_2:=\Hom_\Lambda(M_1,\Lambda)$, \dots, $M_{2n+1}:=\Hom_\Gamma(M_{2n},\Gamma)$,
$M_{2n+2}:=\Hom_\Lambda(M_{2n+1},\Lambda)$, \dots (for $n \geq 0$). Then:
\begin{enumerate}
\item
The recollement admits a ladder of $l$-height exactly $2n+2$ if and only if
$M_j$ is projective as a right $\Gamma$-module for all even $j \leq 2n$,
$M_j$ is projective as a right $\Lambda$-module for all odd $j < 2n+1$ and
$M_{2n+1}$ is not projective as a right $\Lambda$-module.
\item
The recollement admits a ladder of $l$-height exactly $2n+3$ if and only if
$M_j$ is projective as a right $\Gamma$-module for all even $j < 2n+2$, $M_{2n+2}$ is not projective as a right $\Gamma$-module and
$M_j$ is projective as a right $\Lambda$-module for all odd $j \leq 2n+1$.
\end{enumerate}
\end{thm}

\subsection{Abelian ladders from triangulated ladders through
coherent functors}
In this subsection we show how from a recollement of triangulated categories we can obtain a recollement of abelian categories via abelianisation, i.e.\ by taking the category of coherent functors. This method will produce recollements of abelian categories with a ladder.

Recall some basics on coherent functors. Let $\A$ be an additive category. An additive functor $F\colon \A^{\op}\lxr \mathfrak{Ab}$ is called {\bf coherent}, if there exists an exact sequence of the form$\colon$
\[
\xymatrix{
\Hom_{\A}(-,X) \ar[r] & \Hom_{\A}(-,Y) \ar[r]^{ } & F \ar[r] & 0 
}
\] 
where the objects $X$ and $Y$ lie in $\A$. We denote by $\smod{\A}$ the category of coherent functors over $\A$. Recall that a map $X\lxr Y$ is a 
weak kernel of $Y\lxr Z$ if the following sequence is exact$\colon$
\[
\xymatrix{
\Hom_{\A}(-,X) \ar[r] & \Hom_{\A}(-,Y) \ar[r]^{ } & \Hom_{\A}(-,Z)  
}
\]
The category of coherent functors $\smod{\A}$ is abelian if and only if $\A$ has weak kernels. Moreover, the category $\smod{\A}$ has enough projectives and the Yoneda embedding $\mathsf{Y}_{\A}\colon \A\lxr\smod{\A}$, $A\mapsto \Hom_{\A}(-,A)$, induces an equivalence between $\A$ and $\Proj(\smod{\A})$ (when $\A$ has split idempotents). For more details on coherent functors we refer to the work of Auslander \cite{coherent, Auslander:queen}. 

Let $\T$ be a triangulated category. Since $\T$ has weak kernels, the category of coherent functors $\smod{\T}$ is abelian. The latter category is also known as the abelianisation of $\T$, see for example \cite[Appendix A]{Krause:Localizationtheory}.

Let $\me\colon \T\lxr \V$ be a triangle functor between triangulated categories. Then by the universal property of the Yoneda embedding, there is a unique exact functor $\me_{\mathsf{coh}}\colon \smod\T\lxr \smod\V$ such that the following diagram is commutative$\colon$
\[
\xymatrix@C=4em{
\T \ar[r]^{\mathsf{Y}_{\T} \ \ } \ar[d]_{\me} & \smod{\T} 
\ar[d]^{\me_{\mathsf{coh}}} \\
\V  \ar[r]^{\mathsf{Y}_{\V} \ \ } & \smod{\V} }
\]
\begin{minipage}{0.4\textwidth}
Consider now a recollement $\mathsf{R}_{\mathsf{tr}}(\U,\T,\V)$ of triangulated categories$\colon$
\end{minipage}
\begin{minipage}{0.6\textwidth}
\[
\xymatrix@C=0.5cm{
\U \ar[rrr]^{\mathsf{i}} &&& \T \ar[rrr]^{\mathsf{e}}  \ar @/_1.5pc/[lll]_{\mathsf{q}}  \ar
 @/^1.5pc/[lll]^{\mathsf{p}} &&& \V
\ar @/_1.5pc/[lll]_{\mathsf{l}} \ar
 @/^1.5pc/[lll]^{\mathsf{r}}
 }
\]
\end{minipage}
\begin{minipage}{0.4\textwidth}
Then it easily follows (see \cite[Lemma~A.3]{Krause:Localizationtheory}) that there is an adjoint triple$\colon$
\end{minipage}
\begin{minipage}{0.6\textwidth}
\[
\xymatrix@C=0.5cm{
\smod{\T} \ar[rrr]^{\me_{\mathsf{coh}}}  &&& \smod{\V}
\ar @/_1.5pc/[lll]_{\ml_{\mathsf{coh}}} \ar
 @/^1.5pc/[lll]^{\mr_{\mathsf{coh}}} } 
\]
\end{minipage}
Also, the functors $\ml_{\mathsf{coh}}$ and $\mr_{\mathsf{coh}}$ are fully faithful since the functors $\ml$ and $\mr$ are fully faithful, respectively. 

\begin{prop}
\label{proprecocoherentfunctors}
Let $\mathsf{R}_{\mathsf{tr}}(\U,\T,\V)$ be a recollement of triangulated categories. Then:
\begin{enumerate}
\item The abelianisation of $\mathsf{R}_{\mathsf{tr}}(\U,\T,\V)$ gives rise to a recollement of abelian cateogries
\begin{equation}
\label{recolcohfunctors}
\xymatrix@C=0.5cm{
\Ker{\me_{\mathsf{coh}}} \ar[rrr]^{\inc} &&& \smod\T \ar[rrr]^{\me_{\mathsf{coh}}}  \ar @/_1.5pc/[lll]_{}  \ar
 @/^1.5pc/[lll]^{} &&& \smod\V
\ar @/_1.5pc/[lll]_{\ml_{\mathsf{coh}}} \ar
 @/^1.5pc/[lll]^{\mr_{\mathsf{coh}}}
 }
\end{equation}

\item If 
\begin{equation}
\label{completerecolcohfunctors}
\xymatrix@C=0.5cm{
\smod\U \ar[rrr]^{\mi_{\mathsf{coh}}} &&& \smod\T \ar[rrr]^{\me_{\mathsf{coh}}}  \ar @/_1.5pc/[lll]_{\mq_{\mathsf{coh}}}  \ar
 @/^1.5pc/[lll]^{\map_{\mathsf{coh}}} &&& \smod\V
\ar @/_1.5pc/[lll]_{\ml_{\mathsf{coh}}} \ar
 @/^1.5pc/[lll]^{\mr_{\mathsf{coh}}}
 }
\end{equation}
is a recollement of abelian categories, then $\mathsf{R}_{\mathsf{tr}}(\U,\T,\V)$ splits.

\item If $\mathsf{R}_{\mathsf{tr}}(\U,\T,\V)$ admits a ladder of $\ml$-height $n$, resp. $\mr$-height $m$, then the recollement $(\ref{recolcohfunctors})$ admits a ladder of $\ml$-height $n$, resp. $\mr$-height $m$, in the sense of Definition~\ref{defnladder}.
\end{enumerate}
\begin{proof}
(i) This follows from the above discussion and Remark~\ref{remadjointtriplerecol}.

(ii) It is easy to check that the diagram $(\ref{completerecolcohfunctors})$ satisfies all conditions of a recollement of abelian categories except that $\Image \mi_{\mathsf{coh}}=\Ker \me_{\mathsf{coh}}$. This, in particular, means that the sequence of abelian categories $0\lxr \smod\U\lxr \smod\T\lxr \smod\V\lxr 0$ is not, in general, exact. Let us assume now that $(\ref{completerecolcohfunctors})$ is a recollement. 

By Remark \ref{remdefrecol}, for a functor $F$ in 
$\smod{\T}$ there is an exact sequence 
\[
 \xymatrix{
 \ml_{\mathsf{coh}} \me_{\mathsf{coh}}(F) \ar[r]^{ } &
  F  \ar[r]^{ } & \mi_{\mathsf{coh}} \mq_{\mathsf{coh}}(F) \ar[r] & 0  }.
\]
For $F=\Hom_{\T}(-,T)$ there are isomorphisms 
\[
 \ml_{\mathsf{coh}} \me_{\mathsf{coh}}(F)\cong \Hom_{\T}(-,\ml\me(T)) \ \ \ \text{and} \ \ \ \mi_{\mathsf{coh}} \mq_{\mathsf{coh}}(F)\cong \Hom_{\T}(-, \mi\mq(T)).
\]
This yields the exact sequence
\begin{equation}
\label{exactcoherentseq}
 \xymatrix{
 \Hom_{\T}(-,\ml\me(T)) \ar[r]^{ } &
  \Hom_{\T}(-,T)  \ar[r]^{ } & \Hom_{\T}(-, \mi\mq(T))\ar[r] & 0  }.
\end{equation}
Since $\mathsf{R}_{\mathsf{tr}}(\U,\T,\V)$ is a recollement of triangulated categories, there is a canonical triangle $\ml\me(T)\lxr T\lxr \mi\mq(T)\lxr \ml\me(T)[1]$. Note that the maps in $(\ref{exactcoherentseq})$ are induced by the adjunction morphisms of the canonical triangle. Since the sequence $(\ref{exactcoherentseq})$ is exact for all $X$ in $\T$, it follows by Yoneda's Lemma that the morphism $\mi\mq(T)\lxr \ml\me(T)[1]$ is zero. This implies that the canonical triangle splits. Similarly, the exact sequence 
\[
 \xymatrix{
0\ar[r] & \mi_{\mathsf{coh}} \map_{\mathsf{coh}}(F) \ar[r]^{ } &
  F  \ar[r]^{ } & \mr_{\mathsf{coh}} \me_{\mathsf{coh}}(F)   }
\]
yields the canonical triangle $\mi\map(T)\lxr T\lxr \mr\me(T)\lxr \mi\map(T)[1]$ splits. Thus, the recollement $\mathsf{R}_{\mathsf{tr}}(\U,\T,\V)$ splits. 
\begin{minipage}{0.3\textwidth}
(iii) Assume that $\mathsf{R}_{\mathsf{tr}}(\U,\T,\V)$ admits a ladder of $\ml$-height $n$, or a ladder of $\mr$-height $m$ (in the sense of \cite{AKLY1}). This
gives  the diagram$\colon$
\end{minipage}
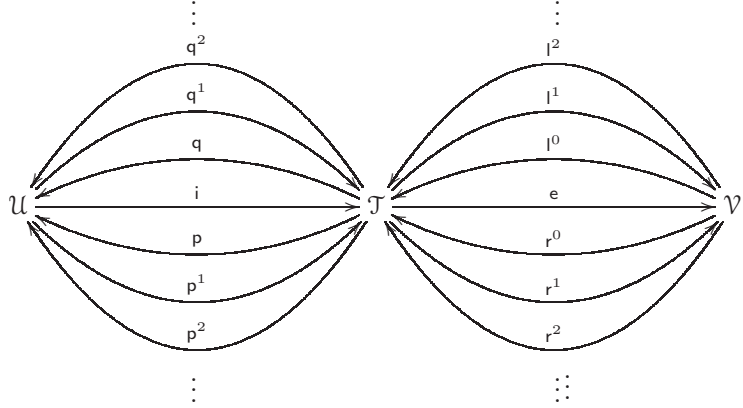
\begin{minipage}{0.7\textwidth}
\[
\xymatrix@C=0.5cm{
 &&& \vdots \ \ \ \ \ \ \ \ \ \ \ \ \ \ \ \ \ \ \ \ \ \ \ \ \ \ \  \ \ \  \ \  \ \ \ \ \ \ \ \ \vdots &&& \\
 &&&  &&& \\
\U \ar[rrr]^{\mathsf{i}} \ar @/_3.0pc/[rrr]^{{\map^1}} \ar @/^3.0pc/[rrr]^{\mq^1}  &&& \T \ar @/_4.5pc/[lll]_{\mq^2}\ar[rrr]^{\mathsf{e}} \ar @/^3.0pc/[rrr]^{\mathsf{l}^1}   \ar @/_1.5pc/[lll]_{\mathsf{q}} \ar @/_3.0pc/[rrr]^{\mathsf{r}^1}  \ar
 @/^1.5pc/[lll]_{\mathsf{p}} \ar @/^4.5pc/[lll]_{\map^2} &&& \V \ar @/_4.5pc/[lll]_{\mathsf{l}^2} \ar @/^4.5pc/[lll]_{\mathsf{r}^2}
\ar @/_1.5pc/[lll]_{\mathsf{l}^0} \ar
 @/^1.5pc/[lll]_{\mathsf{r}^0} \\
 &&& &&& \\
 &&& \vdots \ \ \ \ \ \ \ \ \ \ \ \ \ \ \ \ \ \ \ \ \ \ \ \ \ \ \  \ \ \  \ \  \ \ \ \ \ \ \ \ \vdots &&& \\
 }
\]
\end{minipage}
\begin{minipage}{0.3 \textwidth}
which either goes up $n$ steps or goes down $m$ steps. Then from (i), there is a recollement of abelian categories with a ladder of $\ml$-height $n$ ($n-1$ left adjoints of $\me_{\mathsf{coh}}$)$\colon$
\end{minipage}
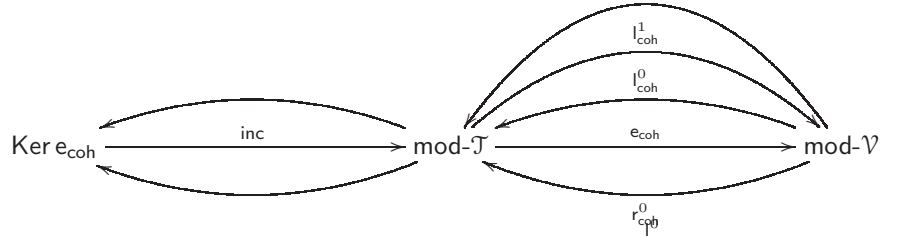
\begin{minipage}{0.7 \textwidth}
\vspace{-1cm}
\[
\xymatrix@C=0.5cm{
 &&& \ \ \ \ \ \ \ \ \ \ \ \ \ \ \ \ \ \ \ \ \ \ \ \ \  \ \ \  \ \  \ \ \ \ \ \ \ \ \ \ \ \ \vdots &&& \\
 &&&  &&& \\
\Ker\mathsf{e}_{\mathsf{coh}} \ar[rrr]^{\inc} &&& \smod\T \ar[rrr]^{\mathsf{e}_{\mathsf{coh}}} \ar @/^3.0pc/[rrr]^{\mathsf{l}^1_{\mathsf{coh}}}   \ar @/_1.5pc/[lll]_{}  \ar
 @/^1.5pc/[lll]^{} &&& \smod\V \ar @/_4.5pc/[lll]_{\mathsf{l}^2_{\mathsf{coh}}} 
\ar @/_1.5pc/[lll]_{\mathsf{l}^0_{\mathsf{coh}}} \ar
 @/^1.5pc/[lll]^{\mathsf{r}^0_{\mathsf{coh}}} \\
 }
 \]
\end{minipage}
\begin{minipage}{0.2 \textwidth}
or a recollement with a ladder of $\mr$-height $m$ ($m-1$ right adjoints of $\me_{\mathsf{coh}}$)$\colon$
\end{minipage}
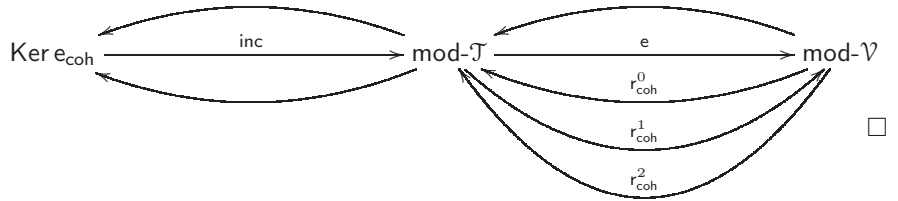
\begin{minipage}{0.28 \textwidth}
\vspace{-0.5cm}
\[
\ \ \ \ \ \ \ \ \ \ \ \ \ \ \ \ \ \ \ \ \ \ \ \ \ \xymatrix@C=0.5cm{
\Ker\mathsf{e}_{\mathsf{coh}} \ar[rrr]^{\inc} &&& \smod\T \ar[rrr]^{\mathsf{e}}   \ar @/_1.5pc/[lll]_{} \ar @/_3.0pc/[rrr]^{\mathsf{r}^1_{\mathsf{coh}}}  \ar
 @/^1.5pc/[lll]^{} &&& \smod\V \ar @/^4.5pc/[lll]_{\mathsf{r}^2_{\mathsf{coh}}}
\ar @/_1.5pc/[lll]_{\mathsf{l}^0} \ar
 @/^1.5pc/[lll]_{\mathsf{r}^0_{\mathsf{coh}}} 
 &&& &&& \\
 &&& &&& \\
 &&& \ \ \ \ \ \ \ \ \ \ \ \ \ \ \ \ \ \ \ \ \ \ \ \ \ \ \ \ \  \ \ \  \ \  \ \ \ \ \ \ \ \ \vdots &&& \\ }
 \]
\end{minipage}
\end{proof}
\end{prop}
\vspace*{-1cm}
\begin{exam}
Let $\mathsf{R}_{\mathsf{tr}}(\U,\T,\V)$ be a recollement of $k$-linear triangulated categories, where $k$ is a field. Assume that $\T$ admits a Serre functor $\mathsf{S}_{\T}$ and let $\mathsf{S}_{\T}^{-1}$ its quasi-inverse. 
An example is the bounded derived category of an algebra
having finite global dimension.
Then from the recollement $\mathsf{R}_{\mathsf{tr}}(\U,\T,\V)$ it follows that $\mathsf{S}_{\T}$ induces Serre functors $\mathsf{S}_{\U}$ and $\mathsf{S}_{\V}$ in $\U$ and $\V$, respectively. Then from \cite{Jorgensen} there is an infinite ladder
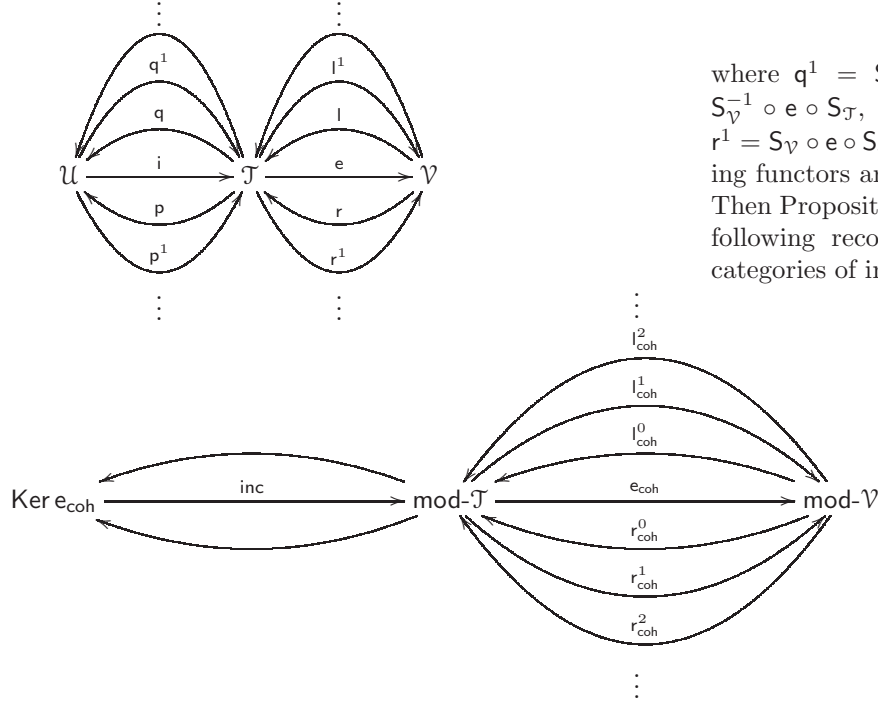
\begin{minipage}{0.7\textwidth}
\vspace*{-0.5cm}
\[
\xymatrix@C=0.5cm{
\U  \ar @/^3.0pc/[rrr]^{\mq^1} \ar @/_3.0pc/[rrr]_{\vdots}^{\map^1} \ar[rrr]^{\mi} &&& \T   \ar @/^3.0pc/[rrr]^{\ml^1} \ar @/_4.5pc/[lll]_{\vdots}  \ar[rrr]^{\me } \ar @/_1.5pc/[lll]_{\mq} \ar @/_3.0pc/[rrr]_{\vdots }^{\mr^1}  \ar @/^1.5pc/[lll]_{\map} &&& \V
\ar @/_1.5pc/[lll]_{\ml} \ar @/_4.5pc/[lll]_{\vdots} \ar
 @/^1.5pc/[lll]_{\mr}   } 
\]
\end{minipage}
\begin{minipage}{0.3\textwidth}
where $\mq^1 = \mathsf{S}_{\T}^{-1}\circ \mi\circ \mathsf{S}_{\U}$, $\ml^{1} = \mathsf{S}_{\V}^{-1}\circ\me\circ \mathsf{S}_{\T}$, $\map^1 = \mathsf{S}_{\T}\circ \mi\circ \mathsf{S}_{\U}^{-1}$, $\mr^{1} = \mathsf{S}_{\V}\circ\me\circ \mathsf{S}_{\T}^{-1}$ and the remaining functors are defined similarly. Then Proposition~\ref{proprecocoherentfunctors} implies the following recollement of abelian categories of infinite ladder$\colon$
\end{minipage}
\vspace*{-1cm}
\[
\xymatrix@C=0.5cm{
 &&& \ \ \ \ \ \ \ \ \ \ \ \ \ \ \ \ \ \ \ \ \ \ \ \ \  \ \ \  \ \  \ \ \ \ \ \ \ \ \ \ \ \ \vdots &&& \\
 &&&  &&& \\
\Ker\mathsf{e}_{\mathsf{coh}} \ar[rrr]^{\inc} &&& \smod\T \ar @/_3.0pc/[rrr]^{\mathsf{r}^1_{\mathsf{coh}}} \ar[rrr]^{\mathsf{e}_{\mathsf{coh}}} \ar @/^3.0pc/[rrr]^{\mathsf{l}^1_{\mathsf{coh}}}   \ar @/_1.5pc/[lll]_{}  \ar
 @/^1.5pc/[lll]^{} &&& \smod\V \ar @/^4.5pc/[lll]_{\mathsf{r}^2_{\mathsf{coh}}} \ar @/_4.5pc/[lll]_{\mathsf{l}^2_{\mathsf{coh}}} 
\ar @/_1.5pc/[lll]_{\mathsf{l}^0_{\mathsf{coh}}} \ar
 @/^1.5pc/[lll]_{\mathsf{r}^0_{\mathsf{coh}}} \\
&&& &&& \\
 &&& \ \ \ \ \ \ \ \ \ \ \ \ \ \ \ \ \ \ \ \ \ \ \ \ \ \ \ \ \  \ \ \  \ \  \ \ \ \ \ \ \ \ \vdots &&& \\ } 
 \]
\end{exam}

\section{Deriving recollements with ladders}

Given a recollement of abelian categories $(\A,\B,\C)$ with ladders,
recollements of
triangulated categories can be constructed that involve derived categories or
singularity categories of the given abelian categories. This proves
in particular Theorem A.

In the sequel, we need the following standard lemma, see for instance \cite{Maltsiniotis}.

\begin{lem}
\label{exactadjoint}
Let $\A$ and $\B$ be abelian categories and assume that there is an adjoint pair of exact functors $(F,G)$, i.e. $\xymatrix{
F\colon \A  \ar@<-.7ex>[r]_-{} & \B\colon G \ar@<-.7ex>[l]_-{} }$, between $\A$ and $\B$. Then there is an adjoint pair $(F,G)$
between the unbounded derived categories of $\A$ and $\B$ which restricts
to the bounded derived categories. Moreover, if $G\colon \B\lxr \A$ is fully faithful, then the induced functor $G\colon \mathsf{D}(\B)\lxr \mathsf{D}(\A)$ is also fully faithful.
\end{lem}

Let $\T$ be a triangulated category. Given a triangulated subcategory $\X$ of
$\T$, the Verdier quotient $\T/\X$ is known to be a triangulated category, and
there is a quotient functor $q\colon \T\lxr \T/\X$. Let $T$ be an object
of $\T$. Then $q(T)\iso 0$ if and only if $T$ is a direct summand of an object in $\X$. When $\X$ is a thick triangulated subcategory of $\T$, the kernel $\Ker{q}$ of $q$ coincides with $\X$, that is, $0\lxr \X\lxr \T\lxr \T/\X\lxr 0$ is an exact sequence of triangulated categories.

\begin{lem}
\label{lempropertiestrianglefunctors}
Let $F\colon \T\lxr \mathcal{S}$ be a triangle functor between triangulated categories which has a right adjoint functor $G\colon \mathcal{S}\lxr \T$.
\begin{enumerate}
\item \textnormal{(\!\!\cite[Proposition 1.5 and 1.6]{BK})} Assume that the functor $G$ is fully faithful. Then the functor $F$ induces a triangle equivalence between the Verdier quotient $\T/\Ker{F}$ and $\mathcal{S}$.

\item \textnormal{(\!\!\cite[Lemma 2.1]{Or})} Let $\X$ and $\Y$ be triangulated subcategories of $\T$ and $\mathcal{S}$, respectively, such that $F(\X)\subseteq \Y$ and $G(\Y)\subseteq \X$. Then $F$ induces a triangle functor $\T/\X\lxr \mathcal{S}/\Y$  and $G$ induces a right adjoint ${mathcal{S}/\Y\lxr \T/\X}$. If $G$ is fully faithful, then the induced right adjoint $\mathcal{S}/\Y\lxr \T/\X$ is also fully faithful.

\item \textnormal{(\!\!\!\cite[Section 2.1]{CPS}, \cite[Section 2]{M})} Assume that the functor $F$ is fully faithful and has a left adjoint $H\colon \mathcal{S}\lxr \T$, i.e.\ $(H,F,G)$ is an adjoint triple. Then there is a recollement of triangulated categories$\colon$
\[
\xymatrix@C=0.5cm{
 \T \ar[rrr]^{F} &&& \mathcal{S} \ar[rrr]^{}  \ar @/_1.5pc/[lll]_{H}  \ar
 @/^1.5pc/[lll]^{G} &&& \mathcal{S}/\Image{F}
\ar @/_1.5pc/[lll]_{} \ar
 @/^1.5pc/[lll]^{}
 }
\]
such that $\Ker{H}={^{\bot}(\Image{F})}\simeq \mathcal{S}/\Image{F}\simeq (\Image{F})^{\bot}=\Ker{G}$.
\end{enumerate}
\end{lem}

\begin{lem}
\label{lemgettingrecolVerdierquotient}
Let $(\U,\T,\V)$ be a recollement of triangulated categories
\[
\xymatrix@C=0.5cm{
\U \ar[rrr]^{\mathsf{i}} &&& \T \ar[rrr]^{\mathsf{e}}  \ar @/_1.5pc/[lll]_{\mathsf{q}}  \ar
 @/^1.5pc/[lll]^{\mathsf{p}} &&& \V
\ar @/_1.5pc/[lll]_{\mathsf{l}} \ar
 @/^1.5pc/[lll]^{\mathsf{r}}
 }
 \]
Assume that $\X$, respectively $\Y$, is a triangulated subcategory of $\T$, respectively $\V$, such that $\ml(\Y)\subseteq \X$, $\me(\X)\subseteq \Y$ and $\mr(\Y)\subseteq \X$. Then there exists a recollement of triangulated categories
\[
\xymatrix@C=0.5cm{
\Ker{\me} \ar[rrr]^{} &&& \T/\X \ar[rrr]^{\mathsf{e}}  \ar @/_1.5pc/[lll]_{}  \ar
 @/^1.5pc/[lll]^{} &&& \V/\Y
\ar @/_1.5pc/[lll]_{\mathsf{l}} \ar
 @/^1.5pc/[lll]^{\mathsf{r}}
 }
\]
\begin{proof} Lemma~\ref{lempropertiestrianglefunctors} (ii) implies that $(\ml, \me, \mr)$ induces an adjoint triple
$(\ml, \me, \mr)$ between $\T/\X$ and  $\V/\Y$ with $\mr$ is fully faithful. Since $\me$ is a Verdier quotient functor, it follows that $\ml$ is fully faithful.
\end{proof}
\end{lem}

Before we produce a recollement of triangulated categories, involving derived or singularity categories, we need to recall the notion of torsion pairs in abelian categories. Let $\A$ be an abelian category. A torsion pair in $\A$ is a pair $(\X,\Y)$ of full subcategories such that (i) $\Hom_{\A}(\X,\Y)=0$, and (ii) for all objects $A\in \A$, there is an exact sequence $0 \lxr X_A \lxr A \lxr Y^A \lxr 0$
where $X_A\in \X$ and $Y^A\in \Y$.

\begin{prop}
\label{lempropertiesladder}
Let $(\A,\B,\C)$ be a recollement of abelian categories.
\begin{enumerate}
\item Assume that $\B$ is a cocomplete abelian category (that is, small
coproducts exist) and that the recollement $(\A,\B,\C)$ admits a ladder of $\ml$-height two. Then the pair $(\Ker{\ml^1},(\Ker{\ml^1})^{\bot})$ is a torsion pair in $\B$.

\item  Assume that the recollement $(\A,\B,\C)$ admits a ladder of $\ml$-height three. Then there is a recollement of abelian categories
\[
\xymatrix@C=0.5cm{
\Ker{\ml^1} \ar[rrr]^{} &&& \B \ar[rrr]^{\ml^1}  \ar @/_1.5pc/[lll]_{}  \ar
 @/^1.5pc/[lll]^{} &&& \C
\ar @/_1.5pc/[lll]_{\mathsf{l}^2} \ar
 @/^1.5pc/[lll]^{\ml^0}
 }
\]
\noindent and $(\ml^1, \ml^0, \me)$ induces a recollement of triangulated categories
\[
\xymatrix@C=0.5cm{
\mathsf{D}(\C) \ar[rrr]^{\ml^0} &&& \mathsf{D}(\B)\ar[rrr]^{} \ar @/_1.5pc/[lll]_{\ml^1}  \ar @/^1.5pc/[lll]^{\me} &&& \mathsf{D}_{\A}(\B)
\ar @/_1.5pc/[lll]_{} \ar
 @/^1.5pc/[lll]_{}
 }
\]
which restricts to the bounded derived categories.

\item Assume that $\B$ is a complete abelian category (that is, small products
exist) and that the recollement $(\A,\B,\C)$ admits a ladder of $\mr$-height two. Then the pair $({^{\bot}(\Ker{\mr^1})},\Ker{\mr^1})$ is a torsion pair in $\B$.

\item  Assume that the recollement $(\A,\B,\C)$ admits a ladder of $\mr$-height three. Then there is a recollement of abelian categories
\[
\xymatrix@C=0.5cm{
\Ker{\mr^1} \ar[rrr]^{} &&& \B \ar[rrr]^{\mr^1}  \ar @/_1.5pc/[lll]_{}  \ar
 @/^1.5pc/[lll]^{} &&& \C
\ar @/_1.5pc/[lll]_{\mathsf{r}^0} \ar
 @/^1.5pc/[lll]^{\mr^2}
 }
\]
\noindent and  $(\me, \mr^0, \mr^1)$ induces a recollement of triangulated categories
\[
\xymatrix@C=0.5cm{
\mathsf{D}(\C) \ar[rrr]^{\mr^0} &&& \mathsf{D}(\B)\ar[rrr]^{} \ar @/_1.5pc/[lll]_{\me}  \ar @/^1.5pc/[lll]^{\mr^1} &&& \mathsf{D}_{\A}(\B)
\ar @/_1.5pc/[lll]_{} \ar
 @/^1.5pc/[lll]_{}
 }
\]
which restricts to the bounded derived categories.
\end{enumerate}
\begin{proof}
(i) It suffices to show that $\Ker{\ml^1}$ is a torsion class, i.e. it is closed under quotient objects, coproducts and extensions. Since $(\ml^1,\ml^{0})$ is an adjoint pair, the functor $\ml^1$ is right exact and therefore $\Ker{\ml^1}$ is closed under quotient objects. Assume that $0\lxr B_1\lxr B\lxr B_2\lxr 0$ is an exact sequence in $\B$ with $B_1, B_2$ in $\Ker{\ml^1}$. Applying $\ml^1$ shows that $B$ also lies in $\Ker{\ml^1}$, that is, $\Ker{\ml^1}$ is closed under extensions. Let $B_i, i\in I$, be a family of objects in $\B$ which lie in $\Ker{\ml^1}$. Since $\ml^1$ is a left adjoint, it preserves coproducts and therefore $\ml^1(\amalg_{i\in I} B_i)\iso \amalg_{i\in I}\ml^1(B_i)=0$. Hence, $\Ker{\ml^1}$ is closed under coproducts. We infer that $(\Ker{\ml^1},(\Ker{\ml^1})^{\bot})$ is a torsion pair.

(ii) Since there is the adjoint triple $(\ml^2,\ml^1,\ml^0)$ and $\ml^0$ is fully faithful, Remark~\ref{remdefrecol} implies the existence of a recollement $(\Ker{\ml^1},\B,\C)$ of abelian categories.

By Lemma~\ref{exactadjoint}, the exact adjoint triple  $(\ml^1, \ml^0, \me)$ between $\B$ and $\C$ induces an adjoint triple $(\ml^1, \ml^0, \me)$ between $\mathsf{D}(\B)$ and $\mathsf{D}(\C)$. Since there is an exact sequence $0\lxr\A\stackrel{\mi}{\lxr}\B\stackrel{e}{\lxr}\C\lxr 0$, $\Ker\me$ is equivalent to $\mathsf{D}_{\A}(\B)$, by \cite[Theorem 3.2]{M} .
By Lemma~\ref{lempropertiestrianglefunctors} there exists a recollement of triangulated categories
\[
\xymatrix@C=0.5cm{
\mathsf{D}(\C) \ar[rrr]^{\ml^0} &&& \mathsf{D}(\B)\ar[rrr]^{} \ar @/_1.5pc/[lll]_{\ml^1}  \ar @/^1.5pc/[lll]^{\me} &&& \mathsf{D}_{\A}(\B)
\ar @/_1.5pc/[lll]_{} \ar
 @/^1.5pc/[lll]_{}
 }
\]

Part (iii) and (iv) follow dually.
\end{proof}
\end{prop}

The  main result of this section contains Theorem~\ref{derivedcat} of the Introduction:

\begin{thm}
\label{mainthmsingladder}
Let $(\A,\B,\C)$ be a recollement of abelian categories.
\begin{enumerate}
\item Assume that $(\A,\B,\C)$ has a ladder of $\ml$-height three. Then there exists a triangle equivalence
\[
\xymatrix{
  \mathsf{D}_{\textsf{sg}}(\B)/\Ker{\ml^1} \ \ar[r]^{ \  \ \ \simeq} & \
  \mathsf{D}_{\textsf{sg}}(\C) }
\]
and  a recollement of triangulated categories
\[
\xymatrix@C=0.5cm{
 \mathsf{D}(\C)\ar[rrr]^{\ml^0} &&& \mathsf{D}(\B) \ar[rrr]^{}  \ar @/_1.5pc/[lll]_{\mathsf{l}^1}  \ar
 @/^1.5pc/[lll]^{\me} &&& \mathsf{D}_{\A}(\B)
\ar @/_1.5pc/[lll]_{} \ar
 @/^1.5pc/[lll]^{}
 }
\]
which restricts to the bounded derived categories.

\item Assume that $(\A,\B,\C)$ has a ladder of $\ml$-height three and $\mr$-height two. Then $(\ml^1, \ml^0, \me)$ induces an adjoint triple between $\mathsf{D}_{\textsf{sg}}(\B)$ and $\mathsf{D}_{\textsf{sg}}(\C)$ and there exists
a recollement of triangulated categories
\[
\xymatrix@C=0.5cm{
 \mathsf{D}_{\textsf{sg}}(\C)\ar[rrr]^{\ml^0} &&& \mathsf{D}_{\textsf{sg}}(\B) \ar[rrr]^{}  \ar @/_1.5pc/[lll]_{\mathsf{l}^1}  \ar
 @/^1.5pc/[lll]^{\me} &&& \Ker{\ml^1}
\ar @/_1.5pc/[lll]_{} \ar
 @/^1.5pc/[lll]^{}
 }
\]

\item Assume that $(\A,\B,\C)$ has a ladder of $\ml$-height four. Then
there exists a recollement of triangulated categories
\[
\xymatrix@C=0.5cm{
\Ker{\ml^1} \ar[rrr]^{} &&& \mathsf{D}_{\textsf{sg}}(\B) \ar[rrr]^{\ml^1}  \ar @/_1.5pc/[lll]_{}  \ar
 @/^1.5pc/[lll]^{} &&& \mathsf{D}_{\textsf{sg}}(\C)
\ar @/_1.5pc/[lll]_{\mathsf{l}^2} \ar
 @/^1.5pc/[lll]^{\ml^0}
 }
\]
If the given ladder has $\mr$-height two, then the recollement  
$(\Ker{\ml^1}, \mathsf{D}_{\textsf{sg}}(\B), \mathsf{D}_{\textsf{sg}}(\C))$ has a 
ladder of height two.

\item Assume that $(\A,\B,\C)$ has a ladder of $\mr$-height three. Then there exists a triangle equivalence
\[
\xymatrix{
  \mathsf{D}_{\textsf{sg}}(\B)/\Ker{\me} \ \ar[r]^{ \  \ \ \simeq} & \
  \mathsf{D}_{\textsf{sg}}(\C) }
\]
\end{enumerate}
\end{thm}
\begin{proof}
(i) Since $(\ml^2,\ml^1)$, $(\ml^1,\ml^0)$ and $(\ml^0,\me)$ are adjoint pairs, the functors $\ml^1\colon \B\lxr \C$ and $\ml^0\colon \C\lxr \B$ are exact and preserve projective objects. So, $\ml^1(\mathsf{K}^{\mathsf{b}}(\Proj{\B}))\subseteq \mathsf{K}^{\mathsf{b}}(\Proj{\C})$ and $\ml^0(\mathsf{K}^{\mathsf{b}}(\Proj{\C}))\subseteq \mathsf{K}^{\mathsf{b}}(\Proj{\B})$. As $\me$ is exact,  Lemma~\ref{exactadjoint} implies
that $(\ml^1,\ml^0, \me)$ is an adjoint triple with  $\ml^0$ fully faithful between $\mathsf{D}(\B)$ and $\mathsf{D}(\C)$, and this triple restricts also to the bounded derived categories. By Lemma~\ref{lempropertiestrianglefunctors} (ii),
$(\ml^1,\ml^0)$ is an adjoint pair of functors between $\mathsf{D}_{\textsf{sg}}(\B)$ and $\mathsf{D}_{\textsf{sg}}(\C)$ with the induced functor $\ml^0\colon \mathsf{D}_{\textsf{sg}}(\C)\lxr \mathsf{D}_{\textsf{sg}}(\B)$ being fully faithful. Hence, Lemma~\ref{lempropertiestrianglefunctors} (i) shows that the triangulated categories $\mathsf{D}_{\textsf{sg}}(\B)/\Ker{\ml^1}$ and $\mathsf{D}_{\textsf{sg}}(\C)$ are equivalent.

By \cite[Theorem 3.2]{M}, there is an exact sequence $0\lxr \mathsf{D}_{\A}^{b}(\B)\lxr \mathsf{D}^{b}(\B)\stackrel{\me}{\lxr}\mathsf{D}^{b}(\C)\lxr 0$ of triangulated categories. Therefore, $\Ker{\me}$ is equivalent to $\mathsf{D}_{\A}^{b}(\B)$. The first part of the proof provides an adjoint triple $(\ml^1,\ml^0,\me)$ between $\mathsf{D}^{b}(\C)$ and $\mathsf{D}^{b}(\B)$, where the triangle functor $\ml^0$ is fully faithful. Then Lemma~\ref{lempropertiestrianglefunctors} (iii) gives the desired recollements.

(ii) Since the $\mr$-height of $(\A,\B,\C)$ is two, $\mr^0$ is exact and therefore $\me\colon \B\lxr \C$ preserves projective objects. As in part (i), $(\ml^0, \me)$ is an adjoint pair of functors between $\mathsf{D}_{\textsf{sg}}(\C)$ and $\mathsf{D}_{\textsf{sg}}(\B)$. Then part (i) provides an adjoint triple $(\ml^1,\ml^0,\me)$ between $\mathsf{D}_{\textsf{sg}}(\C)$ and $\mathsf{D}_{\textsf{sg}}(\B)$, where the triangle functor $\ml^0$ is fully faithful. Thus the recollement $(\mathsf{D}_{\textsf{sg}}(\C), \mathsf{D}_{\textsf{sg}}(\B), \Ker{\ml^1})$ follows from
Lemma~\ref{lempropertiestrianglefunctors} (iii).

(iii) Since the $\ml$-height of $(\A,\B,\C)$ is four, there is a recollement of abelian categories $(\Ker{\ml^1},\B,\C)$ where the adjoint triple between $\B$ and $\C$ is $(\ml^2,\ml^1,\ml^0)$, see Lemma~\ref{proppropertiesladder} (ii).
In this case, all functors are exact and preserve projective objects. As in part (i), the adjoint triple can be derived to get a recollement  of triangulated categories $(\Ker{\ml^1},\mathsf{D}^{\mathsf{b}}(\B),\mathsf{D}^{\mathsf{b}}(\C))$. Then since $\ml^2(\mathsf{K}^{\mathsf{b}}(\Proj{\C}))\subseteq \mathsf{K}^{\mathsf{b}}(\Proj{\B})$, $\ml^1(\mathsf{K}^{\mathsf{b}}(\Proj{\B}))\subseteq \mathsf{K}^{\mathsf{b}}(\Proj{\C})$ and $\ml^0(\mathsf{K}^{\mathsf{b}}(\Proj{\C}))\subseteq \mathsf{K}^{\mathsf{b}}(\Proj{\B})$, $(\Ker{\ml^1}, \mathsf{D}_{\textsf{sg}}(\B), \mathsf{D}_{\textsf{sg}}(\C))$ is a recollement by Lemma~\ref{lemgettingrecolVerdierquotient}.

(iv) By (iii), there is an adjoint triple $(\ml^2,\ml^1,\ml^0)$ and by (ii) the functor $\ml^0$ has $\me$ as a right adjoint.

(v) Since $(\mr^0, \mr^1, \mr^2)$ is an adjoint triple, $\mr^0$ is exact, and $\me$ and $\mr^{0}$ preserve projective objects. So by Lemma~\ref{exactadjoint},
$(\me,\mr^0, \mr^1)$ is an adjoint triple with $\mr^0$ fully faithful between $\mathsf{D}(\B)$ and $\mathsf{D}(\C)$, and this restricts to the bounded derived categories. Lemma~\ref{lempropertiestrianglefunctors} (ii) then implies that $(\me,\mr^0)$ is an adjoint pair of functors between $\mathsf{D}_{\textsf{sg}}(\B)$ and $\mathsf{D}_{\textsf{sg}}(\C)$ where the induced functor $\mr^0\colon \mathsf{D}_{\textsf{sg}}(\C)\lxr \mathsf{D}_{\textsf{sg}}(\B)$ is fully faithful. Hence,
by Lemma~\ref{lempropertiestrianglefunctors} (i), the triangulated categories $\mathsf{D}_{\textsf{sg}}(\B)/\Ker{\me}$ and $\mathsf{D}_{\textsf{sg}}(\C)$ are equivalent.
\end{proof}

\section{Torsion pairs arising from ladders}

A technique will be provided to move torsion pairs in abelian categories via adjoint functors and in particular through Giraud\footnote{See \cite[Chapter~X]{Stenstrom} for the original definition which is equivalent to the one given here.} subcategories in a recollement diagram $(\A, \B, \C)$ with ladders, which proves Theorem B and provides another connection with derived categories.

\begin{minipage}{0.8\textwidth}
Recall from \cite[Definition 1.2]{CFM} that an {\em abelian category with a distinguished Giraud subcategory} is the data $(\D, \C, F, G)$ of two abelian categories $\D$ and $\C$ and two functors $F$ and $G$, with $G$ a left adjoint of $F$, such that $G$ is exact and $F$ is fully faithful. In that case $\C$ is called a Giraud subcategory. Co-Giraud subcategories are defined dually.
\end{minipage}
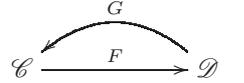
\begin{minipage}{0.2\textwidth}
\[
\xymatrix@C=0.5cm{
\C\ar[rrr]^{F}  &&& \D \ar
 @/_1.5pc/[lll]_{G}
 }
\]
\end{minipage}

\begin{lem}
\label{Gi}
Let $(\A,\B,\C)$ be a recollement of abelian categories, which admits a ladder of $\ml$-height three.
Then $\C$ occurs as a Giraud subcategory of $\B$ in two different ways, in $(\B, \C, \ml^{0}, \ml^{1})$ and in $(\B, \C, \mr^{0}, \me)$. Moreover, $\C$ occurs as a co-Giraud subcategory of $\B$ in two different ways, in $(\B, \C, \ml^{0}, \me)$ and in $(\B, \C, \ml^{2}, \ml^{1})$.
\begin{proof}\ This follows from Proposition~\ref{proppropertiesladder} and
the definition of  Giraud (resp.\ co-Giraud) subcategory.
\end{proof}
\end{lem}

\begin{prop} Let $(\A,\B,\C)$ be a recollement of abelian categories, which admits a ladder of $\ml$-height three. Then the following statements hold.
\begin{enumerate}

\item $(\ml^{1}(\Ker\mq), \ml^{1}(\mi(\A)))$ is a torsion pair in $\C$ if and only if $\ml^{1}\circ\mi=0$. In this case, $\ml^{1}(\Ker\mq)=\C$.

\item $(\ml^{1}(\mi(\A)), \ml^{1}(\Ker\map))$ is a torsion pair in $\C$ if and only if $\map\circ\ml^{0}\circ\ml^{1}(\Ker\map)=0$.

\item $((\ml^{1}(\Ker\mq), \ml^{1}(\mi(\A)))$ is a torsion pair in $\C$ if and only if $\mq\circ\ml^{2}\circ\ml^{1}(\Ker\mq)=0$.

\item $(\ml^{1}(\mi(\A)), \ml^{1}(\Ker\map))$ is a torsion pair in $\C$ if and only if $\me\circ\ml^{2}\circ\ml^{1}\circ \mi=0$.
\end{enumerate}
\begin{proof}
(i) By Lemma~\ref{Gi}, $(\B, \C, \ml^{1}, \ml^{0})$ is a Giraud subcategory of $\B$. Since $(\Ker \mq, \mi(\A))$ is a torsion pair in $\B$ by \cite[Theorem~4.3]{PsaroudVitoria:1}, it follows from \cite[Proposition 3.3]{CFM} that $(\ml^{1}(\Ker \mq), \ml^{1}(\mi(\A)))$ is a torsion pair in $\C$ if and only if $\ml^{0}(\ml^{1}(\mi(\A)))\subseteq \mi(\A)$.
Since $\ml^{0}$ is fully faithful and $\Ker \me=\Image \mi$,  there is an inclusion $\ml^{0}(\ml^{1}(\mi(\A)))\subseteq \mi(\A)$
if and only if $\ml^{1}\circ\mi=0$.  Hence $\ml^{1}(\mi(\A))=0$ and $\ml^{1}(\Ker \mq)=\C$.

(iii) By Lemma~\ref{Gi}, $(\B, \C, \ml^{2}, \ml^{1})$ is a co-Giraud subcategory of $\B$. Since $(\Ker\ \mq, \mi(\A))$ is a torsion pair in $\B$, it follows from \cite[Theorem~3.5]{CFM} that $(\ml^{1}(\Ker\mq), \ml^{1}(\mi(\A)))$ is a torsion pair in $\C$ if and only if $\ml^{2}(\ml^{1}(\Ker\mq))\subseteq \Ker \mq$.
Note that $\ml^{2}(\ml^{1}(\Ker\mq))\subseteq \Ker\mq$ means that $\mq\circ\ml^{2}\circ\ml^{1}(\Ker\mq)=0$.  

The proofs of (ii) and (iv) are similar.
\end{proof}
\end{prop}

Let $\B$ be an abelian category with a torsion pair
$(\mathcal{T}, \mathcal{F})$. Let $\mathcal{H}_{\B}$ be the tilt of $\B$
by the torsion pair  $(\mathcal{T}, \mathcal{F})$. Now we can prove
Theorem~\ref{tilt} of the Introduction:

\begin{proof}[Proof of Theorem B]
Since $(\mathcal{T}, \mathcal{F})$ is a torsion pair in $\B$ such that $\ml^{0}(\ml^{1}(\mathcal{F}))\subseteq \mathcal{F}$ and $\ml^2(\ml^1(\T))\subseteq \T$, it follows from \cite[Proposition 3.3 and 3.8]{CFM} that $(\ml^{1}(\mathcal{T}), \ml^{1}(\mathcal{F}))$ is a torsion pair in $\C$. We denote by $\mathcal{H}_{\C}$ the HRS-tilt of $\C$ by $(\ml^{1}(\mathcal{T}), \ml^{1}(\mathcal{F}))$. This
proves claim (i).

\begin{minipage}{0.4 \textwidth}
Consider the following diagram \\
\xymatrix@C=0.5cm{  \mathsf{D}(\B) \ar[ddd]_{\mathsf{H}^{0}_{\mathsf{D}(\B)}} \ar[rrrr]^{\mathsf{l}^1}   &&&& \mathsf{D}(\C) \ar[ddd]_{\mathsf{H}^{0}_{\mathsf{D}(\C)}}
\ar @/_1.5pc/[llll]_{\mathbb{L}\mathsf{l}^2} \ar
 @/^1.5pc/[llll]^{\mathsf{l}^0} \\ 
  &&&&  \\
  &&&& \\
 \mathcal{H}_{\B} \ar @/_1.5pc/[uuu]_{\epsilon_{\mathcal{H}_{\B}}} \ar[rrrr]^{\mathsf{l}^1_{\mathcal{H}}}   &&&& \mathcal{H}_{\C} \ar @/_1.5pc/[uuu]_{\epsilon_{\mathcal{H}_{\mathcal{C}}}}
\ar @/_1.5pc/[llll]_{\mathsf{l}^2_{\mathcal{H}}} \ar
 @/^1.5pc/[llll]^{\mathsf{l}^0_{\mathcal{H}}} 
 } 
\end{minipage}
\hspace*{1cm}
\begin{minipage}{0.5 \textwidth}
The functors appearing in this diagram are definied as follows. The recollement of abelian categories $(\A,\B,\C)$ admits a ladder of $\ml$-height three. This implies that the functors $\ml^0$ and $\ml^1$ are exact and they induce canonical functors between the corresponding derived categories, still denoted by $\ml^0$ and $\ml^1$. Since $\C$ has enough projectives, we get the derived functor $\mathbb{L}\ml^2\colon \mathsf{D}(\C)\lxr \mathsf{D}(\B)$. As $(\ml^2, \ml^1, \ml^0)$ is an adjoint triple at the level of abelian categories, it is well known that $(\mathbb{L}\ml^2, \ml^1, \ml^0)$ is an adjoint triple at the level of unbounded derived categories. 
By $\epsilon$ we denote the inclusion functor and by $\mathsf{H}^0$ the canonical cohomological functor to the heart. 
\end{minipage}

So there are functors
\[
\ml^2_{\mathcal{H}} = \mathsf{H}^0_{\mathsf{D}(\mathsf{B})}\circ \mathbb{L}\ml^2\circ \epsilon_{\mathcal{H}_{\C}}
\]
\[
\ml^1_{\mathcal{H}} = \mathsf{H}^0_{\mathsf{D}(\mathsf{C})}\circ \ml^1\circ \epsilon_{\mathcal{H}_{\B}}
\]
\[
\ml^0_{\mathcal{H}} = \mathsf{H}^0_{\mathsf{D}(\mathsf{B})}\circ \ml^0\circ \epsilon_{\mathcal{H}_{\C}}
\]
To get the recollement $(\Ker\ml^1_{\mathcal{H}}, \mathcal{H}_{\B}, \mathcal{H}_{\C})$, it suffices by Remark~\ref{remadjointtriplerecol} to show that $(\ml^2_{\mathcal{H}}, \ml^1_{\mathcal{H}}, \ml^0_{\mathcal{H}})$ is an adjoint triple and $\ml^2_{\mathcal{H}}$, equivalently $\ml^0_{\mathcal{H}}$, is fully faithful.

We show that $(\ml^2_{\mathcal{H}}, \ml^1_{\mathcal{H}})$ is an adjoint pair of functors. The \textsf{HRS}-tilt of 
$\C$ by $(\ml^1(\mathcal{T}), \ml^1(\mathcal{F}))$ is the abelian category
\[
\mathcal{H}_{\C}:=\big\{C^{\bullet}\in \mathsf{D}(\mathcal{\C}) \ | \ \mh^{0}(C^{\bullet})\in \ml^1(\mathcal{T}), \mh^{-1}(C^{\bullet})\in \ml^1(\mathcal{F}),
\mh^{i}(C^{\bullet})=0, \forall i> 0, \mh^{i}(C^{\bullet})=0, \forall i<-1 \big\}
\]
and the \textsf{HRS}-tilt of $\B$ by $(\mathcal{T}, \mathcal{F})$ is
\[
\mathcal{H}_{\B}:=\big\{B^{\bullet}\in \mathsf{D}(\mathcal{\B}) \ | \ \mh^{0}(B^{\bullet})\in \mathcal{T}, \mh^{-1}(C^{\bullet})\in \mathcal{F},
\mh^{i}(B^{\bullet})=0, \forall i> 0, \mh^{i}(B^{\bullet})=0, \forall i<-1 \big\}
\]
Denote by $\mathsf{D}^{\geq 0}_{\F}$ the coaisle of the \textsf{HRS} t-structure with heart $\mathcal{H}_{\B}$, and let $\tau^{\geq 0}_{\mathcal{F}}\colon \mathsf{D}(\mathcal{B})\lxr \mathsf{D}^{\geq 0}_{\F}$ be the left adjoint of the inclusion functor $\inc_{\F}\colon \mathsf{D}^{\geq 0}_{\F}\lxr \mathsf{D}(\mathcal{B})$. Since $(\mathbb{L}\ml^2, \ml^1)$ is an adjoint pair, the functor $\mathbb{L}\ml^2$ is right t-exact, i.e. it sends an aisle to an aisle. The same remark is used below for the functor $\ml^1$. Then we have the following sequence of isomorphisms
\begin{eqnarray}
 \Hom_{\mathcal{H}_{\B}}(\ml^2_{\mathcal{H}}(X), Y) &=& \Hom_{\mathcal{H}_{\B}}(\mathsf{H}^0_{\mathsf{D}(\mathsf{B})} \mathbb{L}\ml^2 \epsilon_{\mathcal{H}_{\C}}(X), Y) \nonumber \\
   &\cong & \Hom_{\mathsf{D}(\B)}(\tau^{\geq 0}_{\mathcal{F}} \mathbb{L}\ml^2(X), Y) \nonumber \\
   &\cong & \Hom_{\mathsf{D}(\B)}(\mathbb{L}\ml^2(X), \inc_{\F}(Y)) \nonumber \\
   &\cong & \Hom_{\mathsf{D}(\C)}(X, \ml^1(\inc_{\F}(Y))) \nonumber \\
   &\cong & \Hom_{\mathcal{H}_{\C}}(X, \mathsf{H}^0_{\mathsf{D}(\mathsf{\C})} \ml^1\epsilon_{\mathcal{H}_{\B}}(Y)) \nonumber \\
    &=& \Hom_{\mathcal{H}_{\C}}(X, \ml^1_{\mathcal{H}}(Y))
\end{eqnarray}
Similarly, $(\ml^1_{\mathcal{H}}, \ml^0_{\mathcal{H}})$ is seen to be
an adjoint pair.   

By Remark~\ref{remadjointtriplerecol} it suffices to show that the functor $\ml^0_{\mathcal{H}}$ is fully faithful, or equivalently, the counit map $\ml^1_{\mathcal{H}}\ml^0_{\mathcal{H}}(X)\lxr X$ is a natural isomorphism. But this is easy to verify using the fact that both functors $\ml^1$ and $\ml^0$ are $t$-exact. This completes the proof of (ii). Finally, part (iii) follows by deriving the recollement $(\Ker\ml^1_{\mathcal{H}}, \mathcal{H}_{\B}, \mathcal{H}_{\C})$. Note that since $\ml^1_{\mathcal{H}}$ and $\ml^0_{\mathcal{H}}$ are exact, we get the derived functors between the corresponding derived categories of hearts and since $\mathcal{H}_{\C}$ has enough projectives we can derive the functor $\ml^2_{\mathcal{H}}$.
\end{proof}

\section{Using ladders to transfer Gorenstein properties}

Given a recollement diagram $(\A,\B,\C)$ with ladders, Gorensteinness of $\B$ and $\C$ will be connected, as well as the stable categories of Gorenstein projective objects of $\B$ and $\C$. Moreover, invariance of being Gorenstein projective or injective under certain functors in the recollement will be shown. In
particular, Theorem C will be proved.

\subsection{\bf Gorenstein properties}\
Let $\A$ be an abelian category with enough projective and injective objects. Associated to $\A$ are the following homological invariants\footnote{The invariants $\spli$ and $\silp$ were defined by Gedrich and Gruenberg \cite{GedrichGruenberg}  over any ring $R$. In particular, $\silp{R}$ is defined as the supremum of the injective lengths (dimensions) of projective $R$-modules, and $\spli{R}$ is the supremum of the projective lengths (dimensions) of injective $R$-modules. They were introduced to investigate complete cohomological functors of groups.}:
\[
\spli{\A} = \sup\{{\pd}_{\A}I \mid I\in \Inj{\A} \}
\ \ \text{and} \ \
\silp{\A} = \sup\{{\id}_{\A}P \mid P\in \Proj{\A} \}
\]
The category $\A$ is called {\bf Gorenstein} if $\spli{\A}<\infty$ and $\silp{\A}<\infty$. Moreover, $\A$ is called {\bf n-Gorenstein} if its Gorenstein dimension $\textsf{G}$-$\dime{\A}=\maxx\{\spli{\A}, \silp{\A}\}\leq n$, see \cite[Chapter VII]{BR} for more information on Gorenstein abelian categories.

According to Beligiannis-Reiten \cite[Theorem 2.2, Chapter VII]{BR}, an abelian category is Gorenstein if and only if every object has a finite resolution by Gorenstein projective objects, which are defined as follows:
A complex of projective objects $\textsf{P}^{\bullet}\colon \cdots \lxr P^{-1}\lxr P^{0}\lxr P^{1}\lxr \cdots $ is called \textsf{totally acyclic} if $\textsf{P}^{\bullet}$ and $\Hom_{\A}(\textsf{P}^{\bullet},P)$ are acyclic for every projective object $P$ of $\A$. Then an object $X$ of $\A$ is called \textsf{Gorenstein projective} if $X$ is isomorphic to $\Coker{(P^{-1}\lxr P^{0})}$ for some totally acyclic complex $\textsf{P}^{\bullet}$ of projective objects of $\A$. We denote by $\GProj{\A}$ the full subcategory of Gorenstein projective objects of $\A$. Now let $X$ be an object in $\GProj{\A}$ and $\textsf{P}^{ \, \bullet}$ its totally acyclic complex. Recall from \cite{EJ} that  for every object $Y$ of $\A$ with $\pd_{\A}Y<\infty$ the complex $\Hom_{\A}(\textsf{P}^{ \, \bullet},Y)$ is acyclic.

Gorenstein injective object  can be defined dually. We denote by $\Ginj{\A}$ the full subcategory of Gorenstein injective objects of $\A$. Then
$\A$ is called virtually Gorenstein if $\GProj{\A}^{\perp}=^{\perp}\Ginj{\A}$.

Let $(\A,\B,\C)$ be a recollement of abelian categories, where the categories $\B$ and $\C$ are assumed to have enough projective and injective objects. Recall from \cite{Psaroud:homolrecol} that the $\A$-\textsf{relative global dimension} of $\B$ is defined by $\gd_{\A}\B=\sup\{\pd_{\B}\mi(A) \ | \ A\in \A\}$ (where
$\mi$ as usual is the inclusion functor). In what follows the finiteness of this dimension is needed.

\begin{lem}\textnormal{(\!\!\cite[Proposition $4.4$]{Psaroud:homolrecol})}
\label{lemprojinj}
Let $(\A,\B,\C)$ be a recollement of abelian categories and let $X$ be an object of $\B$. Then the following inequalities hold$\colon$
\begin{enumerate}
\item $\pd_{\B}X \ \leq \ \pd_{\C}\me(X)+\gd_{\A}\B+1$.

\item $\id_{\B}X \ \leq \ \id_{\C}\me(X)+\sup\{\id_{\B}\mi(A) \ | \ A\in \A\}+1$.

\end{enumerate}
\end{lem}

\begin{lem}
\label{glo}
Let $(\A,\B,\C)$ be a recollement of abelian categories.
\begin{enumerate}
\item Assume that the functor $\mr\colon \C\lxr \B$ is exact. Then $\spli{\C} \leq \spli{\B}$. Moreover, if
$\spli{\B}<\infty$, then $\sup\{\pd_{\C}\me(I) \ | \ I\in \Inj{\B}\}<\infty$.

\item Assume that the functor $\mr\colon \C\lxr \B$ is exact, $\sup\{\id_{\B}\mi(A) \ | \ A\in \A\}<\infty$ and $\silp{\C}<\infty$. Then $\silp{\B}<\infty$.

\item Assume that the $\A$-relative global dimension of $\B$ is finite and $\sup\{\pd_{\C}$ $\me(I) \ | \ I\in \Inj{\B}\}<\infty$. Then $\pd_{\B}\mr(\me(P))< \infty$ for every projective object $P$ of $\B$,
and $\spli{\B}<\infty$.

\end{enumerate}
\begin{proof}
(i) Assume that $\spli{\B}=n<\infty$.  Since $\mr\colon \C\lxr \B$ is exact, $\me\colon \B\lxr \C$ preserves projective objects.
Let $I$ be an injective object of $\C$. Then the object $\mr(I)$ is injective in $\B$ and therefore there exists an exact sequence
$0\lxr P_n\lxr \cdots \lxr P_0\lxr \mr(I)\lxr 0$ with $P_i\in \Proj{\B}$. Applying $\me$ yields the exact sequence $0\lxr \me(P_n)\lxr \cdots \lxr \me(P_0)\lxr I\lxr 0$ with $\me(P_i)$ projective.
This implies that $\pd_{\C}I\leq n$. Hence, $\spli{\C} \leq \spli{\B}$.

Let $I$ be an injective object of $\B$.  Then $\spli{\B}=n$ implies $\pd_{\C}\me(I)\leq\pd_{\B}I\leq n$ and therefore $\sup\{\pd_{\C}\me(I) \ | \ I\in \Inj{\B}\}<\infty$.

(ii) Since the functor $\mr\colon \C\lxr \B$ is exact, the functor $\me\colon \B\lxr \C$ preserves projective objects. Assume that  $\sup\{\id_{\B}\mi(A) \ | \ A\in \A\}=n<\infty$ and let $P$ be a projective object of $\B$. By Lemma~\ref{lemprojinj}, $\id_{\B}P \leq \id_{\C}\me(P)+n+1\leq \silp{\C}+n+1$. Therefore, $\silp{\B}<\infty$.

(iii) Assume that $\gd_{\A}\B=\lambda<\infty$. Let $P$ be a projective object of $\B$. The long exact homology sequence of the exact sequence $0\lxr \mi\map(P)\lxr P\xrightarrow{\nu_{P}} \mr\me(P)\lxr \Coker{\nu_P}\lxr 0$, where $\Coker{\nu_P}=\mi(A)$ for some $A$ of $\A$, yields that $\Ext_{\B}^m(\mr(\me(P)),B)=0$ for every $m\geq \lambda+2$ and $B\in \B$. This shows that the projective dimension of $\mr(\me(P))$ is less than or equal to $\lambda+1$.

Suppose that $\sup\{\pd_{\C}\me(I) \ | \ I\in \Inj{\B}\}=\kappa<\infty$. Let $I$ be an injective object of $\B$. Then $\pd_{\C}\me(I)\leq \kappa$ and
Lemma~\ref{lemprojinj} implies $\pd_{\B}\mr\me(I)\leq \kappa+\lambda+1$. From the  exact sequence $0\lxr \mi\map(I)\lxr I\xrightarrow{\nu_I}\mr\me(I)\lxr \Coker{\nu_I}\lxr 0$, where $\Coker{\nu_I}=\mi(A)$ for some object $A$ of $\A$, it follows that $\pd_{\B}I\leq \kappa+\lambda+1<\infty$.
Hence $\spli{\B}<\infty$.
\end{proof}
\end{lem}

\begin{lem}
\textnormal{(\!\!\cite[Lemma 3.2]{LernerOppermann})}
\label{Extadjoint}
Let $\A$ and $\B$ be exact categories, $F\colon \A\lxr \B$ and $G\colon \B\lxr \A$ a pair of exact adjoint functors. Then
for any integer $n\geq 0$, ${\Ext}^{n}_{\A}(A, G(B))\cong {\Ext}^{n}_{\B}(F(A),B)$,
functorial in $A\in \A$ and $B\in \B$.
\end{lem}

The main result of this Section includes Theorem~\ref{Gorentein} of the Introduction:

\begin{thm}
\label{pregproj}
Let $(\A,\B,\C)$ be a recollement of abelian categories.
\begin{enumerate}

\item Assume that the $\A$-relative global dimension of $\B$ is finite and that $(\A,\B,\C)$ admits a ladder of $\mr$-height two. Then the functor $\me\colon \B\lxr \C$ preserves Gorenstein projective objects. Moreover, if $(\A,\B,\C)$ admits a ladder of $\ml$-height two, then the functor $\ml\colon \C\lxr \B$ preserves Gorenstein projective objects, and $\B$ is Gorenstein if and only if $\C$ is Gorenstein.

\item Assume that $(\A,\B,\C)$ has a ladder of $\ml$-height three. Then the functor $\ml^{1}\colon \B\lxr \C$ preserves Gorenstein projective objects and the functor $\me\colon \B\lxr \C$ preserves Gorenstein injective objects. Furthermore, if $\B$ is $n$-Gorenstein, then $\C$ is $n$-Gorenstein.

\item Assume that  $(\A,\B,\C)$ admits a ladder of $\ml$-height four. Then the functor $\ml\colon \C\lxr \B$ preserves Gorenstein injective objects. Moreover, $\me\circ\ml\cong \iden_{\GInj\C}$.

\item Assume that  $(\A,\B,\C)$ admits a ladder of $\mr$-height three. Then the functor $\me\colon \B\lxr \C$ preserves Gorenstein projective objects and the functor $\mr^{1}\colon \B\lxr \C$ preserves Gorenstein injective objects. Furthermore, if $\B$ is Gorenstein, then $\C$ is Gorenstein.

\item Assume that  $(\A,\B,\C)$ admits a ladder of $\mr$-height four. Then the functor $\mr\colon \C\lxr \B$ preserves Gorenstein projective objects. Moreover, $\me\circ\mr\cong {\iden}_{\GProj\C}$.

\item Assume that  $(\A,\B,\C)$ admits a ladder of $\ml$-height two and $\mr$-height three. Then the functor $\mr\colon \C\lxr \B$ preserves Gorenstein injective objects. Moreover, $\mr^{1}\circ\mr\cong {\iden}_{\Ginj\C}$.

\item Assume that  $(\A,\B,\C)$ admits a ladder of $\ml$-height two and  $\mr$-height two. If $\B$ is virtually Gorenstein, then $(\me({\GProj}\B), \me(\mathcal{P}^{<\infty}(\B)), \me({\GInj}\B))$ is a cotorsion triple in $\C$, where $\mathcal{P}^{<\infty}(\B)$ is the full subcategory of $\B$ consisting of objects of finite projective dimension.

\item Assume that  $(\A,\B,\C)$ admits a ladder of $\ml$-height four. If $\B$ is virtually Gorenstein, then $(\ml^{1}, \ml)$ induces an adjoint pair $(\ml^{1}, \ml)$ between $({\GProj}\B)^{\perp}$ and ${^{\perp}({\GInj}\C)}$. Moreover, $\ml^{1}\circ \ml\cong {\iden}_{{^{\perp}({\GInj}\C)}}$.

\item Assume that  $(\A,\B,\C)$ admits a ladder of $\mr$-height four. If $\C$ is virtually Gorenstein, then  $(\mr, \mr^{1})$ induces an adjoint pair $(\mr, \mr^{1})$ between ${^{\perp}({\GInj}\C)}$ and $({\GProj}\B)^{\perp}$. Moreover, $\mr^{1}\circ \mr\cong {\iden}_{{^{\perp}({\GInj}\C)}}$.
\end{enumerate}
\begin{proof}
(i) Let $B$ be a Gorenstein projective object of $\B$. Then there exists an exact complex of projectives $\textsf{P}^{\bullet}\colon \cdots \lxr P^{-2}\lxr P^{-1}\lxr P^0\lxr P^{1}\lxr P^2\lxr \cdots$ such that  $\Coker{(P^{-1}\lxr P^{0})}$ is isomorphic to $B$ and the complex $\Hom_{\B}(\textsf{P}^{\bullet},P)$ is exact for every projective object $P$ of $\B$. Since $\mr$ is exact and $(\me,\mr)$ is an adjoint pair, the complex $\me(\textsf{P}^{\bullet})\colon \cdots \lxr \me(P^{-2})\lxr \me(P^{-1})\lxr \me(P^0)\lxr \me(P^{1})\lxr \me(P^2)\lxr \cdots$ is exact in $\C$ with $\me(P^i)\in \Proj{\C}$ and $\Coker{(\me(P^{-1})\lxr \me(P^{0}))}\iso \me(B)$. Note that since the functor $\ml\colon \C\lxr \B$ preserves projective objects, the category of projectives of $\C$ is equivalent with $\add{\me(\Proj{\B})}$. Thus we have to show that the complex $\cdots\lxr \Hom_{\C}(\me(P^1),\me(P))\lxr \Hom_{\C}(\me(P^0),\me(P))\lxr \Hom_{\C}(\me(P^{-1}),\me(P))\lxr \cdots$ is exact for every $P\in \Proj{\B}$. The adjoint pair $(\me,\mr)$ yields the commutative diagram
\[
\xymatrix@C=0.5cm{
\cdots \ar[r] & \Hom_{\C}(\me(P^1),\me(P)) \ar[r] \ar[d]^{\iso} & \Hom_{\C}(\me(P^0),\me(P)) \ar[r] \ar[d]^{\iso} & \Hom_{\C}(\me(P^{-1}),\me(P)) \ar[d]^{\iso} \ar[r] & \cdots  \\
 \cdots \ar[r] & \Hom_{\B}(P^1,\mr\me(P)) \ar[r] & \Hom_{\B}(P^0,\mr\me(P)) \ar[r] & \Hom_{\B}(P^{-1},\mr\me(P)) \ar[r] & \cdots  }
\]
Then by Lemma~\ref{glo}, the complex  $\Hom_{\B}(\textsf{P}^{\bullet},\mr(\me(P)))$ is exact since the projective dimension of $\mr(\me(P))$ is finite. Hence the complex $\Hom_{\C}(\me(\textsf{P}^{\bullet}),\me(P))$ is exact. So, the object $\me(B)$ is Gorenstein projective in $\C$. Let $C$ be an object of $\C$. Then for the object $\ml(C)$ of $\B$,
there is an exact sequence $0\lxr X_n\lxr \cdots \lxr X_0\lxr \ml(C)\lxr 0$ with $X_i\in \GProj{\B}$, by \cite[Theorem 2.2, Chapter VII]{BR}. Thus, there is an exact sequence $0\lxr \me(X_n)\lxr \cdots \lxr \me(X_0)\lxr C\lxr 0$ with $\me(X_i)$ in $\GProj{\C}$. This shows, again using \cite[Theorem 2.2, Chapter VII]{BR}, that the category $\C$ is Gorenstein.

Conversely, suppose that $\C$ is Gorenstein and let $B$ an object in $\B$. By Remark~\ref{remdefrecol}, there exists an exact sequence
$(*)\colon0\lxr \mi(A)\lxr \ml\me(B)\lxr B \lxr \mi\mq(B)\lxr 0$ with $A\in \A$. Since the $\A$-relative global dimension of $\B$ is finite, it follows that $\mi(A)$ and $\mi\mq(B)$ have finite projective dimension. On the other hand, for the object $\me(B)$, there is an exact sequence $0\lxr Y_n\lxr \cdots \lxr Y_0\lxr \me(B)\lxr 0$ with $Y_i\in \GProj{\C}$. Since  $(\A,\B,\C)$ admits a ladder of $\ml$-height two, the functor $\ml$ is exact and since $\me$ preserve projectives it follows similarly as above that $\ml$ preserves Gorenstein projective objects. Thus, there is an exact sequence $0\lxr \ml(Y_n)\lxr \cdots \lxr \ml(Y_0)\lxr \ml\me(B)\lxr 0$ with $\ml(Y_i)$ in $\GProj{\B}$. From the exact sequence $(*)$ we infer that the object $B$ admits a finite resolution from Gorenstein projectives, and therefore from \cite[Theorem 2.2, Chapter VII]{BR} we conclude that the category $\B$ is Gorenstein.

(ii)  Since $(\A,\B,\C)$ has a ladder of $\ml$-height three, that is, $(\ml^2,\ml^1,\ml)$ is an adjoint triple, $\ml^1$ and $\ml$ are exact and they preserve projective objects. Moreover, $\ml$ and $\me$ preserve injective objects. Let $G$ be a Gorenstein projective object in $\B$. Then there exists a totally acyclic complex of projective objects of $\B\colon$
\[
\mathsf{P}^{\bullet}\colon \
\xymatrix{
  \cdots \ar[r]_{}^{} & P^{-1} \ar[r]^{d^{-1}} & P^0 \ar[r]^{d^0} & P^{1} \ar[r] & \cdots }
\]
that is, this complex is exact with terms in $\Proj{\B}$ such that the complex $\Hom_{\B}(P^{\bullet}, E)$ is still exact for any $E\in \Proj{\B}$ and $G\cong \Image{d^{0}}$.
Applying the functor $\ml^{1}$ yields that $\ml^{1}(\mathsf{P}^{\bullet})$ is an exact sequence of projective objects of $\C$, and also $\Hom_{\C}(\ml^{1}(P^{\bullet}), Q)\cong \Hom_{\B}(P^{\bullet}, \ml(Q))$ is exact for any projective object $Q$ of $\C$. This implies that $\ml^{1}(G)$ is a Gorenstein projective object of $\C$.

Let $G$ be a Gorenstein injective object in $\B$. Then there exists an exact sequence $I^{\bullet}:=\cdots\to I^{i}\xrightarrow{d^{i}} I^{i+1}\to \cdots$ of injective objects of $\B$ such that $\Hom_{\B}(E, I^{\bullet})$ is still exact for any injective object $E\in \B$ with $G\cong {\Image}d^{0}$. Since $\me$ and $\ml$ preserve injective objects, applying $\me$, we get that $\me(I^{\bullet})$ is an exact sequence of injective objects of $\C$, and also $\Hom_{\C}(I, \me(I^{\bullet}))\cong \Hom_{\B}(\ml(I), I^{\bullet})$ is exact for any injective object $I$ of $\C$. Therefore, $\me(G)$ is a Gorenstein injective object of $\C$.

Let $I\in \Inj{\C}$ and $P\in \Proj{\C}$. We claim that $\pd_{\C}I\leq n$ and $\id_{\C}P\leq n$. Since $(\ml^1,\ml)$ is an adjoint pair and $\ml$ is fully faithful, there is an isomorphism $\ml^1(\ml(I))\iso I$ where $\ml(I)\in \Inj{\B}$. Since $\spli{\B}\leq n$ there is an exact sequence $0\lxr P_n\lxr\cdots \lxr P_0\lxr \ml(I)\lxr 0$ with $P_i\in \Proj{\B}$. Then applying the exact functor $\ml^1$ and using that all $\ml^1(P_i)\in \Proj{\C}$, we obtain that $\pd_{\C}I\leq n$. Hence $\spli{\C}\leq n$. Also, since $\silp{\B}\leq n$ there is an exact sequence $(*)\colon 0\lxr \ml(P)\lxr I^0\lxr \cdots\lxr I^n\lxr 0$ where $\ml(P)\in \Proj{\B}$ and $I^i\in \Inj{\B}$. Applying the exact functor $\me$ to the sequence $(*)$, yields the exact sequence $0\lxr P\lxr \me(I^0)\lxr \cdots\lxr \me(I^n)\lxr 0$ where $\me(I^i)$ lies in $\Inj{\C}$. Hence $\id_{\C}P\leq n$ and therefore $\silp{\C}\leq n$. So, $\C$ is $n$-Gorenstein.

(iii) Since $(\A,\B,\C)$ admits a ladder of $\ml$-height four, $\ml^{1}$ and $\ml$ are exact functors preserving injective objects. Let $G$ be a Gorenstein injective object in $\B$. Then there exists an exact sequence $I^{\bullet}:=\cdots\to I^{i}\xrightarrow{d^{i}} I^{i+1}\to \cdots$ of injective objects of $\B$ with $\Hom_{\B}(E, I^{\bullet})$ is still exact for any injective object $E\in \B$
such that $G\cong {\rm Im}d^{0}$. Applying $\ml$, we get that $\ml(I^{\bullet})$ is an exact sequence of injective objects of $\C$, and also $\Hom_{\C}(I, \ml(I^{\bullet}))\cong \Hom_{\B}(\ml^{1}(I), I^{\bullet})$ is exact for any injective object $I$ of $\C$. This implies that $\ml(G)$ is a Gorenstein injective object of $\C$. Furthermore, by (ii), $\me$ preserves Gorenstein injective objects. 
So, $\me\circ\ml\cong {\rm Id}_{{\rm GInj}\C}$.

(iv) Since $(\A,\B,\C)$ admits a ladder of $\mr$-height three, $\mr^{1}$ and $\mr$ are exact and preserve injective objects. Moreover, $\me$ and $\mr$ preserve projective objects. Let $G$ be a Gorenstein projective object in $\B$.
As in the first part of the proof of (ii) it is shown that $\me(G)$ is a Gorenstein projective object of $\C$.

Let $G$ be a Gorenstein injective object in $\B$. Then there exists an exact sequence $I^{\bullet}:=\cdots\to I^{i}\xrightarrow{d^{i}} I^{i+1}\to \cdots$ of injective objects of $\B$ scuh that $\Hom_{\B}(E, I^{\bullet})$ is still exact for any injective object $E\in \B$ with $G\cong \Image{d^{0}}$. Applying $\mr^{1}$ yields $\mr^{1}(I^{\bullet})$ is an exact sequence of injective objects of $\C$, and also $\Hom_{\C}(I, \mr^{1}(I^{\bullet}))\cong \Hom_{\B}(\mr(I), I^{\bullet})$ is exact for any injective object $I$ of $\C$. This implies that $\mr^{1}(G)$ is a Gorenstein injective object of $\C$.

Let $C\in \C$. Then for the object $\mr(C)$ of $\B$, by \cite[Theorem 2.2, Chapter VII]{BR}, there is an exact sequence $0\lxr X_n\lxr \cdots \lxr X_0\lxr \mr(C)\lxr 0$ with $X_i\in \GProj{\B}$. There also exists an exact sequence $0\lxr \me(X_n)\lxr \cdots \lxr \me(X_0)\lxr C\lxr 0$ with $\me(X_i)\in \GProj{\C}$. 
Hence, the category $\C$ is Gorenstein.

(v) Since $(\A,\B,\C)$ admits a ladder of $\mr$-height four, $\mr^{1}$ is an exact functor preserving projective objects and $\mr$ preserves projective objects. Let $G$ be a Gorenstein projective object in $\C$. Then there exists an exact sequence $P^{\bullet}:=\cdots\to P^{i}\xrightarrow{d^{i}} P^{i+1}\to \cdots$ of projective objects of $\C$ with ${\rm Hom}_{\C}(P^{\bullet}, E)$ still exact for any projective object $E\in \C$
such that $G\cong \Image{d^{0}}$. Applying $\mr$, we get that $\mr(P^{\bullet})$ is an exact sequence of projective objects of $\B$, and also $\Hom_{\B}(\mr(P^{\bullet}), Q)\cong {\rm Hom}_{\C}(P^{\bullet}, \mr^{1}(Q))$ is exact for any projective object $Q$ of $\B$. Thus $\mr(G)$ is a Gorenstein projective object of $\B$. Furthermore, by (iv), $\me$ preserves Gorenstein projective objects. Therefore, $\me\circ\mr\cong {\iden}_{\GProj\C}$. The proof of (vi) is dual.

(vii) Since $(\A,\B,\C)$ admits a ladder of $\ml$-height two and  $\mr$-height two, $\ml$ and  $\mr$ are exact.
Since $\B$ is virtually Gorenstein, $(\GProj\B, \mathcal{P}^{<\infty}(\B), \GInj\B)$ is a cotorsion triple. We claim that $(\me(\GProj\B), \me(\mathcal{P}^{<\infty}(\B)))$ is a cotorsion pair in $\C$. Indeed,  let $X$ be in $\Gproj\B$ and $Z$ in $\C$ such that $\Ext^{1}_{\C}(\me(X), Z)=0$. Then by Lemma~\ref{Extadjoint}, $0=\Ext^{1}_{\C}(\me(X), Z)=\Ext^{1}_{\B}(X, \mr(Z))$. Thus $\mr(Z)\in \mathcal{P}^{<\infty}(\B)$. This implies that  $Z\cong \me\mr(Z)\in \me(\mathcal{P}^{<\infty}(\B))$.
Let $Y\in \mathcal{P}^{<\infty}(\B)$ and  $Z\in \C$ such that $\Ext^{1}_{\C}(Z, \me(Y))=0$. Then by Lemma~\ref{Extadjoint},  $0=\Ext^{1}_{\C}(Z, \me(Y))=\Ext^{1}_{B}(\ml(Z), Y)$. Thus $\ml(Z)\in \GProj\B$ and  $Z\cong \me\ml(Z)\in \me(\GProj\B)$. Similarly it is shown that $(\me(\mathcal{P}^{<\infty}(\B)), \me(\GInj\B))$ is a cotorsion pair in $\C$.

(viii) Let $X$ be in $(\GProj\B)^{\perp}$ and $E$ in $\GInj\C$. By Lemma~\ref{Extadjoint} and (iii), for all $i\geq 1$ there is an isomorphism $\Ext_{\C}^{i}(\ml^{1}(X), E)\cong \Ext_{\B}^{i}(X, \ml(E))$. Since $\B$ is virtually Gorenstein, $X\in ^{\perp}(\GInj\B)$ and so $\Ext_{\C}^{i}(\ml^{1}(X), E)=0$ for all ${i\geq 1}$. This implies that $\ml^{1}\colon (\GProj\B)^{\perp}\lxr {^{\perp}(\GInj\C)}$ is well defined. On the other hand, let $Y$ be in ${^{\perp}(\GInj\C)}$ and $I$ in $\GInj\B$. By Lemma~\ref{Extadjoint} and (ii), there are isomorphisms, for all $i\geq 1$, $\Ext_{\B}^{i}(\ml(Y), I)
\cong \Ext_{\C}^{i}(Y, \me(I))=0$. Since $\B$ is virtually Gorenstein, it follows that $\ml(Y)$ lies in $(\GProj\B)^{\perp}$. This implies that $\ml\colon {^{\perp}(\GInj\C)}\lxr (\GProj\B)^{\perp}$ is well defined. Thus $\ml^{1}\circ \ml\cong {\iden}_{{^{\perp}(\GInj\C)}}$. The proof of (ix) is dual.
\end{proof}
\end{thm}

\begin{cor} Let $(\A,\B,\C)$ be a recollement of abelian categories. Assume that  $(\A,\B,\C)$ admits a ladder of $\mr$-height four.
Then there is an upper recollement of triangulated categories
\[
\xymatrix@C=0.5cm{
\mathsf{K}(\GProj{\C}) \ar[rrr]^{\mr}  &&& \mathsf{K}(\GProj{\B}) \ar[rrr] \ar
 @/_1.5pc/[lll]_{\me}&&& \Ker{\me} \ar
 @/_1.5pc/[lll]
 }
\]
\begin{proof} By Theorem~\ref{pregproj} (iv) and (v), $(\me, \mr)$ induces an adjoint pair between $\GProj{\B}$ and $\GProj{\C}$ with $\mr$ fully faithful. Furthermore, $(\me, \mr)$ induces an adjoint pair $(\me, \mr)$ between $\mathsf{K}(\GProj{\B})$ and $\mathsf{K}(\GProj{\C})$ with $\mr\colon \mathsf{K}(\GProj{\C})\lxr \mathsf{K}(\GProj{\B})$ fully faithful. Then we get the desired upper recollement of triangulated categories. 
\end{proof}
\end{cor}

\subsection{\bf Stable categories of Gorenstein projective modules}
Finally, given a recollement diagram $(\A,\B,\C)$, the stable categories of
Gorenstein projective objects of $\B$ and of $\C$ are related.

\begin{prop}
\label{propfunctorsbetweenstableGproj}
Let $(\A,\B,\C)$ be a recollement of abelian categories
\begin{enumerate}
\item Assume that $(\A,\B,\C)$ admits a ladder of $l$-height two and $r$-height three. Then $(\ml, \me)$ induces an adjoint pair $(\ml, \me)$ between $\underline{\GProj}\C$ and $\underline{\GProj}\B$ with $\ml\colon \underline{\GProj}\C\lxr \underline{\GProj}\B$ fully faithful.

\item \label{mr-four}
Assume that $(\A,\B,\C)$ admits a ladder of $\mr$-height four.
Then $(\me, \mr)$ induces an adjoint pair $(\me, \mr)$ between $\underline{\GProj}\B$ and $\underline{\GProj}\C$ with $\mr\colon \underline{\GProj}\C\lxr \underline{\GProj}\B$ fully faithful.
\end{enumerate}
\end{prop}
\begin{proof}
(a) By the proof of Theorem~\ref{pregproj} (iv), the functor $\me$ preserves projective objects and Gorenstein projective objects. This means that $\me$ induces a triangle functor $\me\colon \underline{\GProj}\B\lxr \underline{\GProj}\C$. Now we claim that $\ml$ preserves Gorenstein projective objects. Indeed, since $\me$ is exact, $\ml$ preserves projective objects. Since $\ml$ has a left adjoint, $\ml$ is exact. Let $G$ be a Gorenstein projective object in $\C$. Then there exists a totally acyclic complex of projective objects of $\C\colon$
\[
\mathsf{P}^{\bullet}\colon \
\xymatrix{
  \cdots \ar[r]_{}^{} & P^{-1} \ar[r]^{d^{-1}} & P^0 \ar[r]^{d^0} & P^{1} \ar[r] & \cdots }
\]
that is, this complex is exact with terms in $\Proj{\C}$ such that the complex $\Hom_{\B}(P^{\bullet}, E)$ is still exact for any $E$ in $\Proj{\C}$ and $G\cong \Image{d^{0}}$.
Applying the functor $\ml$ we get that $\ml(\mathsf{P}^{\bullet})$ is an exact sequence of projective objects of $\B$, and also $\Hom_{\B}(\ml(P^{\bullet}), Q)\cong \Hom_{\B}(P^{\bullet}, \me(Q))$ is exact for any projective object $Q$ of $\B$. This implies that $\ml(G)$ is a Gorenstein projective object of $\B$.
It follows that $\ml$ induces a triangle functor $\ml\colon \underline{\GProj}\C\lxr \underline{\GProj}\B$.

Finally we show that $(\ml, \me)$ is an adjoint pair between $\underline{\GProj}\C$ and $\underline{\GProj}\B$ with $\ml$ fully faithful. Let $X$ an object in $\GProj\C$ and $Y$ an object in $\GProj\B$. If $f\colon X\lxr \me(Y)$ factors through $P\in \Proj\C$, then the morphism $\ml(f)\colon \ml(X)\lxr \ml \me(Y)$ factors through $\ml(P)$. By Remark~\ref{remdefrecol} (v) there is an exact sequence $0\lxr \Ker\mu_{Y}\lxr \ml\me(Y)\xrightarrow{\mu_{Y}} Y\lxr \mi\mq(Y)\lxr 0$, where $\mu\colon \ml\me\lxr \iden_{\B}$ is the counit of $(\ml, \me)$ and $\Ker\mu_{Y}=\mi(A)$ for some $A$ in $\A$. Applying $\Hom_{\B}(\ml(P), -)$, it follows that $\Hom_{\B}(\ml(P), \ml\me(Y))\cong \Hom_{\B}(\ml(P), Y)$. This implies that $\mu_{Y}\circ \ml(f)\colon \ml(X)\lxr Y$ factors through $\ml(P)$. On the other hand, if $g\colon \ml(X)\lxr Y$ factors through $Q\in \Proj\B$, then $\me(g)\colon \me \ml(X)\lxr \me(Y)$ factors through $\me(Q)$. By $\nu_{X}\colon X\cong \me \ml(X)$, where $\nu\colon \iden_{\C}\lxr \me\circ \ml$ is the unit of $(\ml, \me)$, we get that $\me(g)\circ \nu_{X}\colon X\lxr \me(Y)$ factors through $\me(Q)$. Thus, $\underline{\Hom}_{\B}(\ml(X), Y)\cong \underline{\Hom}_{\C}(X, \me(Y))$.

The proof of (ii) is similar, using Theorem~\ref{pregproj} (iv) and (v), and
showing that $(\me, \mr)$ is an adjoint pair between $\underline{\GProj}\B$ and $\underline{\GProj}\C$.
\end{proof}

\begin{exam}
Let $A$ be an Artin algebra and $e$ an idempotent of $A$ such that $AeA$ is a projective left  $A$-module. Let
$\Lambda =
         \begin{pmatrix}
           A & AeA & AeA \\
           A & A & AeA \\
           A & A & A
         \end{pmatrix}$ and  $e=
         \begin{pmatrix}
           0 & 0 & 0 \\
           0 & 0 & 0 \\
           0 & 0 & 1
         \end{pmatrix}$. Then $(e\Lambda\otimes_{\Lambda}-, \Hom_{A}(e\Lambda,-))$ induces an adjoint pair $(e\Lambda\otimes_{\Lambda}-, \Hom_{A}(e\Lambda,-))$ between $\underline{\Gproj}\Lambda$ and $\underline{\Gproj}A$ with $\Hom_{A}(e\Lambda,-)\colon \underline{\Gproj}A\lxr \underline{\Gproj}\Lambda$ fully faithful.

This follows by applying part (ii) of Proposition~\ref{propfunctorsbetweenstableGproj} with $\Gamma: =\begin{pmatrix}
           A/AeA & 0 \\
           A/AeA & A/AeA \\
          \end{pmatrix}$. Then by Proposition~\ref{ideal} (ii), $(\smod{\Gamma}, \smod{\Lambda}, \smod{A})$ is a recollement which admits a ladder of
$\mr$-height at least four.
\end{exam}

\begin{exam}
Let $\Lambda$ be an Artin algebra and let $\Delta_{(0,0)}=\bigl(\begin{smallmatrix}
\Lambda & \Lambda \\
\Lambda & \Lambda
\end{smallmatrix}\bigr)$ be the Morita context ring as in Example~\ref{tofDelta}. Then the triangle functor $\mt_{2}\colon \underline{\Gproj}\Lambda\lxr \underline{\Gproj}\Delta_{(0,0)}$, $\mt_2(X)=(X,X,0,\iden_{X})$, is fully faithful.
Indeed, the recollement $(\smod\Lambda, \smod\Delta_{(0,0)}, \smod\Lambda)$ admits a ladder of $\ml$-height $\infty$ and
$\mr$-height $\infty$ (see \cite[Remark~4.8, Example~4.9]{GaoPsaroudakis2}), as follows:
\[
\xymatrix@C=0.5cm{
\smod\Lambda \ar[rrr]^{\mz_2} &&& \smod\Delta_{(0,0)} \ar @/_3.0pc/[rrr]^{\mathsf{U}_2} \ar @/^3.0pc/[rrr]^{\mathsf{U}_2} \ar[rrr]^{\mU_1} \ar @/_1.5pc/[lll]_{}  \ar @/^1.5pc/[lll]^{} &&& \smod\Lambda \ar @/_4.5pc/[lll]_{\mt_2} \ar @/^4.5pc/[lll]^{\mt_1}
\ar @/_1.5pc/[lll]_{\mt_1} \ar
 @/^1.5pc/[lll]_{\mt_2}
 },
 \]
Then by Proposition~\ref{propfunctorsbetweenstableGproj}, there are triangle functors $\mU_{1}\colon \underline{\Gproj}\Delta_{(0,0)}\lxr \underline{\Gproj}\Lambda$ and $\mt_{2}\colon \underline{\Gproj}\Lambda \lxr \underline{\Gproj}\Delta_{(0,0)}$ such that $({\mU}_{1}, \mt_{2})$ is an adjoint pair with $\mt_{2}$ fully faithful.
\end{exam}

\end{document}